%% file: main.tex
\documentclass[12pt]{article}
\usepackage[utf8]{inputenc}
\usepackage[english]{babel}
\usepackage{t1enc}
\usepackage{amsmath}
\usepackage{amsthm}
\usepackage{amsfonts}  
\usepackage{amssymb}
\usepackage{color,graphicx,transparent}
\usepackage{hyperref}
\usepackage{todonotes}

\allowdisplaybreaks[1]

\newcommand{\bit}{\begin{itemize}}
\newcommand{\eit}{\end{itemize}}
\newcommand{\real}{\mathbb{R}}
\newcommand{\N}{\mathbb{N}}
\newcommand{\Z}{\mathbb{Z}}

\newtheorem{theorem}{Theorem}[section]
\newtheorem{definition}[theorem]{Definition}
\newtheorem{lemma}[theorem]{Lemma}

\newtheorem{remark}[theorem]{Remark}

\title{An extension operator for Sobolev spaces with mixed weights}

\author{
Markus Hansen\footnote{Philipps-University Marburg,  FB12 Mathematics and Computer Science, Hans-Meerwein Stra\ss{}e, Lahnberge, 35032 Marburg, Germany. Email: \href{mailto:hansen@mathematik.uni-marburg.de}{hansen@mathematik.uni-marburg.de}}
\thanks{The work of this author has been supported by Deutsche Forschungsgemeinschaft (DFG), Grant No. 360/22-1.}
\qquad
Cornelia Schneider\footnote{Friedrich-Alexander University Erlangen-Nuremberg, Applied Mathematics III, Cauerstr. 11, 91058 Erlangen, Germany. Email: \href{mailto:schneider@math.fau.de}{schneider@math.fau.de}}\ \thanks{The work of this author has been supported by Deutsche Forschungsgemeinschaft (DFG), Grant No. SCHN 1509/1-2.}
\qquad 
Flóra O. Szemenyei\footnote{\emph{Corresponding author}. Friedrich-Alexander University Erlangen-Nuremberg, Applied Mathematics III, Cauerstr. 11, 91058 Erlangen, Germany. Email: \href{mailto:szemenyei@math.fau.de}{szemenyeif@math.fau.de}\vspace{0.2cm}}
}

\begin{document}
\maketitle

\footnotetext{{\em Math Subject Classifications. Primary:}   46E35, 47A57.   {\em Secondary:} 35B65. }


\begin{abstract}
We provide an  extension operator for weighted Sobolev spaces  on bounded polyhedral cones $K$ involving a mixture of weights, which measure the distance to the vertex and the edges of the cone, respectively. 
Our results are based on Stein's extension operator \cite{Stein} for Sobolev spaces and generalize \cite{Han15}. 
\end{abstract}



\section{Introduction}
In this article we construct an extension operator for weighted Sobolev spaces $V_{\beta,\delta}^{l,p}(K)$ on  polyhedral cones $K$ involving mixed weights, which measure the distance to the vertex and the edges of the cone as displayed by the parameters $\beta$ and $\delta$, respectively. 
Our construction is based on the original  extension operator from Stein \cite{Stein} for classical Sobolev spaces on sufficiently smooth domains, which was generalized  in \cite{Han15} by Hansen to the setting of  weighted Sobolev spaces involving one weight function only. \\
 Weighted Sobolev spaces are particularly  important when it comes to regularity theory for solutions of partial differential equations (PDEs) on non-smooth domains: in this case  singularities of the solutions at the boundary  may occur which diminish the Sobolev regularity, cf. \cite{JK95},   but can be compensated with suitable weight functions. In particular, spaces with mixed weights are needed when studying stochastic PDEs as is demonstrated in  \cite{CLK19, Cio20}. 
 The weighted Sobolev spaces we are interested in  can be seen as generalizations of the so-called Kondratiev spaces which appeared in the 60s in \cite{Kon1, Kon2} and were studied in detail in \cite{SMCW}. 
Later on  more general spaces were considered by  
Kufner, S\"andig \cite{Ku},  Babuska, Guo \cite{BG97}, Nistor, Mazzucato \cite{NistorMazzucato}, and Costabel, Dauge, Nicaise \cite{CDN10}, to mention at least a few contributions in this context.\\ 
 The weighted Sobolev spaces $V_{\beta,\delta}^{l,p}(K)$ with parameters  $l\in \N_0$, $1\leq p<\infty$, $\beta\in \real$, and $\delta=(\delta_1,\ldots, \delta_n)\in \real^n$, were  introduced and studied in detail by Maz'ya, Rossmann in \cite{MR10} and contain all measurable functions $u$ such that the norm 
\[
\|u| V_{\beta,\delta}^{l,p}(K)\|:=
\displaystyle \bigg( \int_{K}\sum_{|\alpha|\leq l} 
\rho_{0}(x)^{p(\beta-l+|\alpha|)}\prod_{k=1}^{n} \bigg( \dfrac{r_k(x)}{\rho_{0}(x)}  \bigg)^{p(\delta_k-l+|\alpha|)} |\partial^{\alpha}u(x)|^p\: dx
\bigg)^{1/p}
\] 
is finite. Here  $\rho_0$ denotes the regularized distance of a point $x$ to the vertex of the cone $K$ and $r_k(x)$ the (regularized) distances to the respective edges $M_k$ of the cone. It is precisely this mixture of the different weights involved and their interplay, which allows for an even finer description when investigating the singularities of the solutions of PDEs on polyhedral cones or even general  domains of polyhedral type.\\
Our construction of the extension operator in the weighted Sobolev spaces $V_{\beta,\delta}^{l,p}$ will be done in three steps: 
 First we show that Stein's definition of an extension operator for special Lipschitz domains remains bounded in the framework of weighted Sobolev spaces  denoted by $V_{\delta}^{l,p}$  where the weight involved (displayed by the parameter $\delta \in \real$) measures the distance  to the $x_1$-axis. In particular, the result holds for  unbounded wedges (which are special Lipschitz domains) where the weight in the respective Sobolev spaces now measures the distance to the edge of the wedge. In the next step we establish an extension operator for a fixed layer $K'$ (bounded away from the vertex) of the cone $K$ in the spaces $V_{\beta,\delta}^{l,p}$ with mixed weights.  The construction uses the first  result, since the layer can be covered by finitely many diffeomorphic images of bounded wedges.  
 Finally, in a third step we decompose the cone into dyadic layers $K_j$, where $j\in \mathbb{Z}$, and construct a family of extension operators for each  layer, which is uniformly bounded  in the norms of $V_{\beta,\delta}^{l,p}$ with respect to $j$. Afterwards we glue together  these operators properly
and use  the localization of the norm of $V_{\beta,\delta}^{l,p}$, which yields  an  extension  operator for the whole cone. \\
Our motivation for these kinds of studies is as follows: in a forthcoming paper we will use regularity estimates for   solutions of elliptic and parabolic PDEs on polyhedral cones in the spaces $V_{\beta,\delta}^{l,p}(K)$  in order to 
study the smoothness $\alpha>0$ of these solutions in the specific scale 
\begin{equation}\label{adaptivityscale}
    B^{\alpha}_{\tau, \tau}(K), \qquad   \frac{1}{\tau}=\frac{\alpha}{d}+\frac 1p, 
\end{equation} 
of Besov spaces.  
It is well known that the regularity $\alpha$ in these spaces determines the approximation order that can be achieved by adaptive  approximation  schemes, see \cite{DDV97} for further details on this subject. 
In order to justify the use of adaptive algorithms (compared to non-adaptive ones) for the PDEs under investigation,  one needs to know that the Besov regularity $\alpha$  of the solutions is sufficiently high.  
Our main tool here will be an embedding for  the weighted Sobolev spaces $V_{\beta,\delta}^{l,p}$  into the scale of Besov spaces \eqref{adaptivityscale}. In this context the  extension operator we construct below will  help to improve the upper bound of the smoothness $\alpha$ considerably, leading to a far better approximation rate of the adaptive algorithms compared to  \cite{DS08}. However,  the results for the extension operator established in this paper are of interest on their own, which is why we decided to publish them  separately. \\

   The paper is organized as follows. In Section 2 we introduce the weighted Sobolev spaces on polyhedral cones  and study some of  their relevant  properties. Afterwards, in Section 3, we construct an extension operator for these spaces in three steps as explained above. 

\section{Preliminaries}

\paragraph{Notation}

We start by collecting some general notation used throughout the paper.

As usual,   $\N$ stands for the set of all natural numbers, $\N_0=\mathbb N\cup\{0\}$, $\Z$ denotes the integers, and 
$\real^d$, $d\in\N$, is the $d$-dimensional real Euclidean space with $|x|$, for $x\in\real^d$, denoting the Euclidean norm of $x$. 
Let $\N_0^d$ be the set of all multi-indices, $\alpha = (\alpha_1, \ldots,\alpha_d)$ with 
$\alpha_j\in\N_0$ and $|\alpha| := \sum_{j=1}^d \alpha_j$. For partial derivatives $\partial^\alpha f=\frac{\partial^{|\alpha|}f}{\partial x^\alpha}$ we will occasionally also write $f_{x^\alpha}$. Furthermore, $B(x,r)$  is the open ball 
of radius $r >0$ centered at $x$, and for a measurable set $M\subset\real^d$ we denote by $|M|$ its Lebesgue measure.

We denote by  $c$ a generic positive constant which is independent of the main parameters, but its value may change from line to line. 
The expression $A\lesssim B$ means that $ A \leq c\,B$. If $A \lesssim
B$ and $B\lesssim A$, then we write $A \sim B$.  


Throughout the paper 'domain' always stands for an open and connected set.  
The test functions on a domain $\Omega$   are denoted by  $C^{\infty}_0(\Omega)$.    
Let $L_p(\Omega)$, $1\leq p\leq \infty$, be the Lebesgue spaces on $\Omega$ as usual.  
We denote by ${C}(\Omega)$  the space of all bounded  continuous functions $f:\Omega\rightarrow \mathbb{R}$  
and  ${C}^k(\Omega)$, $k\in \N_0$, is the space of all functions $f\in {C}(\Omega)$ such that $\partial^{\alpha}f\in {C}(\Omega)$ for all $\alpha\in\N_0$ with $|\alpha|\leq k$, 
endowed with the norm $\sum_{|\alpha|\leq k}\sup_{x\in \Omega}|\partial^{\alpha}f(x)|$.\\
Let $m \in \N_0$ and  $1\le p \leq \infty$. Then
 $W^{m}_p(\Omega)$ denotes the standard \textcolor{black}{$L_p$-Sobolev spaces of order $m$}  on the domain $\Omega$, equipped with the norm
$
\|\, u \, |W^m_{p}(\Omega)\| := \Big(\sum_{|\alpha|\leq m} \int_\Omega |\partial^\alpha u(x)|^p\,dx\Big)^{1/p}. 
$ 



\paragraph{Polyhedral cone}\label{polyhedralcone}
We mainly consider function spaces defined on polyhedral cones in the sequel. Let 
$$K:=\{ x \in \mathbb{R}^3: \: 0<|x|<\infty,\: x/|x|\in \Omega  \}$$
be an infinite cone in $ \mathbb{R}^3 $ with vertex at the origin. Suppose that the boundary $\partial K$ consists of the vertex $x=0$, the edges (half lines) $M_1,...,M_n$, and smooth faces $\Gamma_1,...,\Gamma_n$. Hence, $\Omega\cap S^2$ is a domain of polygonal type on the unit sphere $S^2$ with sides $\Gamma_k \cap S^2$. 
Moreover, we consider the bounded polyhedral cone $\tilde{K}$ obtained via  truncation \[\tilde{K}:=K\cap B(0,r).\]
\begin{figure}[ht] 
\begin{minipage}{0.5\textwidth}             
    \def\svgwidth{185pt}    
    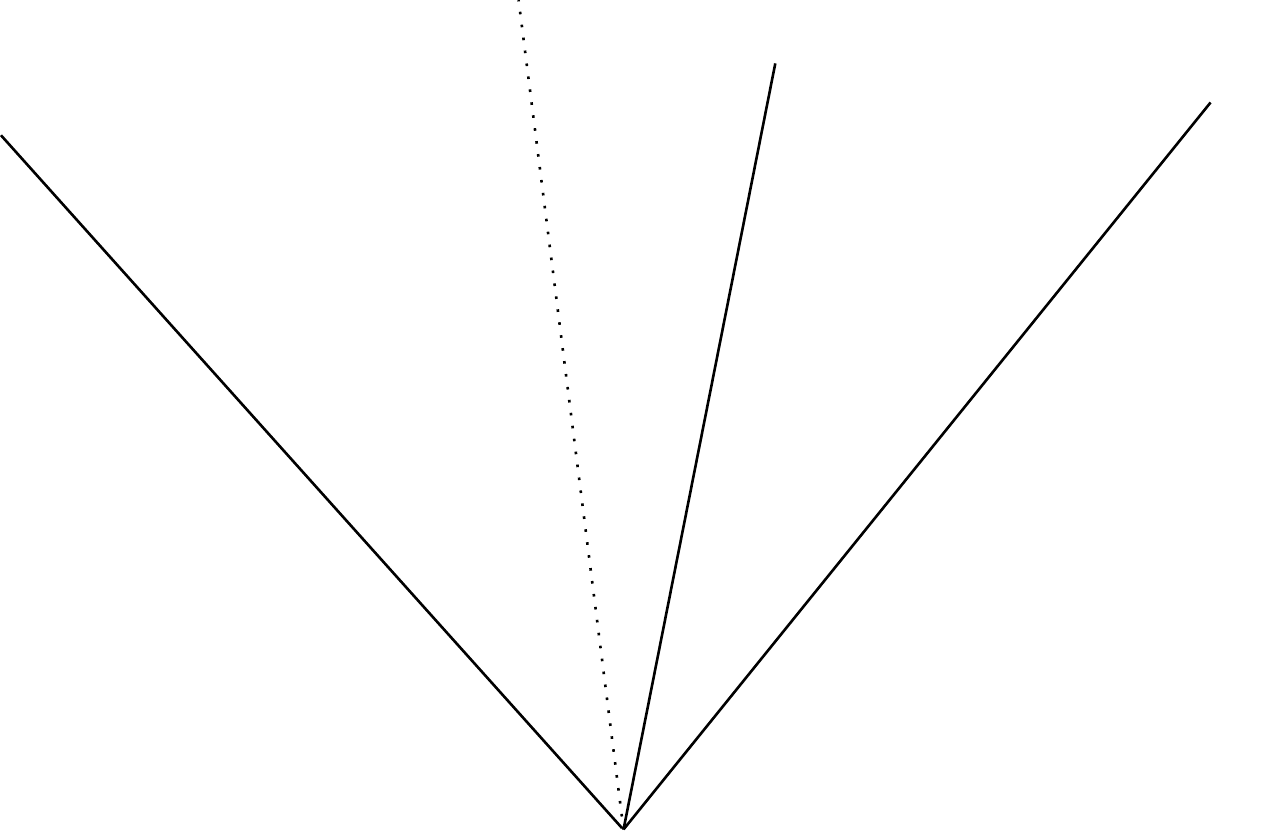  
    \caption{Infinite polyhedral cone $K$} 
    \label{fig:inf}
    \end{minipage}\hfill \begin{minipage}{0.5\textwidth}  
           \centering              
    \def\svgwidth{130pt}    
    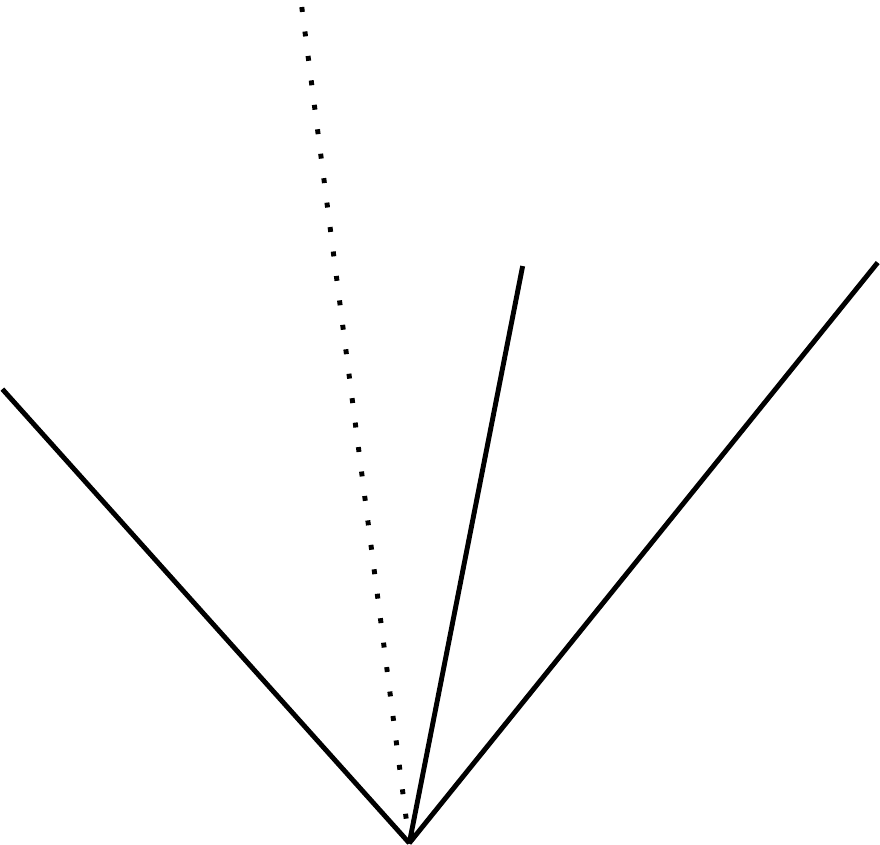
    \caption{Bounded polyhedral cone $\tilde{K}$}
    \label{fig:bound}  
    \end{minipage}\\
\end{figure}

The singular points of the polyhedral cone $K$ are those $x\in \partial K$ for which for any $\varepsilon>0$ the set $\partial K \cap B(x, \varepsilon)$ is not smooth, i.e., the vertex $0$ and the edges $M_1,...,M_n$. 
When we consider the bounded cone $\tilde{K}$  we omit the non-smooth points induced by the truncation and also consider $S=\{0\}\cup M_1 \cup...\cup M_n$, which in this case is not the entire singularity set.\\

\subsection{Weighted Sobolev spaces and their properties}
\label{Vspace}

\begin{definition}
Let $K$ be a (bounded or unbounded) polyhedral cone in $\mathbb{R}^3$ and  $S=\{0\}\cup M_1 \cup...\cup M_n$. Then the space $ V_{\beta,\delta}^{l,p}(K,S)$ is defined as the closure of the set $$C_*^{\infty}(K,S):=\{ u|_K: \: u\in C_0^{\infty}(\mathbb{R}^3\setminus S) \}$$ 
with respect to the norm 
\begin{align}
\label{norm}
&\|u| V_{\beta,\delta}^{l,p}(K,S)\|:=
\displaystyle \bigg( \int_{K}\sum_{|\alpha|\leq l} 
\rho_{0}(x)^{p(\beta-l+|\alpha|)}\prod_{k=1}^{n} \bigg( \dfrac{r_k(x)}{\rho_{0}(x)}  \bigg)^{p(\delta_k-l+|\alpha|)} |\partial^{\alpha}u(x)|^p\: dx
\bigg)^{1/p}, 
\end{align}
where $  \beta \in \mathbb{R}$, $l\in\mathbb{N}_0$, $\delta=(\delta_1,...,\delta_n)\in \mathbb{R}^n$, $1\leq p< \infty$, $\rho_0(x)=\mathrm{dist}(0,x)$ denotes the distance to the vertex, and $r_j(x):= \mathrm{dist} (x,M_j) $ the distance to the edge $M_j$.
\end{definition}

\begin{remark}
We collect some remarks and properties concerning the weighted Sobolev spaces $ V_{\beta,\delta}^{l,p}(K)$.

\begin{itemize}
\item In the sequel if there is no confusion we  omit $S$ from the notation and write shortly $V_{\beta,\delta}^{l,p}(K)$.
\item The space $ V_{\beta,\delta}^{l,p}(K)$ is a Banach space for $1\leq p <\infty$. The proof can be derived from a generalized result, see \cite[p. 18, Theorem 3.6]{Kuf}.

    \item $V_{\beta,\delta}^{l,p}(K) \subset L_p(K)$ for $l\geq\beta $ and $l\geq\delta_k,\: k=1,...,n$.
      
    \item We have the following embeddings which one obtains easily from the definition of the spaces $ V_{\beta,\delta}^{l,p}(K)$: 
    $$V_{\beta,\delta}^{l,p}(K)\subset V_{\beta-1,\delta-1}^{l-1,p}(K)\subset ...\subset V_{\beta-l,\delta-l}^{0,p}(K). $$

      \item There exist regularized versions of the distance functions $\rho_0$ and $r_1$,...,$r_n$. We denote them by $\widetilde{\rho}_0, \widetilde{r}_1,...,\widetilde{r}_n$. According to \cite[Theorem VI.2.2, p.171]{Stein}, these functions defined on $\overline{K}$ have nonnegative function values and are infinitely often differentiable. Moreover, there exist positive constants $A_0, B_0,C_{\alpha,0}$ such that 
      \begin{align*}
      A_0 \rho_0(x)\leq \widetilde{\rho}_0(x)\leq B\rho_0(x),\quad x\in K,
       \end{align*}
    and for all $\alpha\in\mathbb{N}_0^n$, 
    \begin{align*}
    |\partial^{\alpha}\widetilde{\rho}_0(x)|\leq C_{\alpha,0}\rho_0^{1-|\alpha|}(x),\quad x\in K.
    \end{align*}
    One has similar results for $\widetilde{r_i},\: i=1,...,n$.
    
    \item Replacing $\rho_0,r_1,...,r_n$ by $\widetilde{\rho}_0,\widetilde{r}_1,...,\widetilde{r}_n$ in the norm of $V_{\beta,\delta}^{l,p}(K)$, one can prove that the operator
    \begin{align}
    \label{aug12}
    T_{\beta',\delta'}& :V_{\beta,\delta}^{l,p}(K)  \longrightarrow V_{\beta-\beta',\delta-\delta'}^{l,p}(K),\qquad 
    u \mapsto \rho_0^{\beta'}(x) \prod_{k=1}^{n}\left(\frac{r_k(x)}{\rho_0(x)}   \right)^{\delta_k'}u,
    \end{align}
     is an isomorphism, where $\delta-\delta'=(\delta_1-\delta_1',...,\delta_n-\delta_n')$. The inverse operator is given by 
   $  ( T_{\beta',\delta'})^{-1}= T_{-\beta',-\delta'}.$ 
   \item   A function $\varphi\in C^l(K)$ is a pointwise multiplier in $V_{\beta,\delta}^{l,p}(K)$, i.e., for  
     $u\in V_{\beta,\delta}^{l,p}(K)$ we have
   \begin{align}
   \label{okt15d}
   \|\varphi u|V_{\beta,\delta}^{l,p}(K)\|\leq c \|u|V_{\beta,\delta}^{l,p}(K)\|.
   \end{align}
   \item The weighted Sobolev spaces $ V_{\beta,\delta}^{l,p}(K)$ are refinements of the Kondratiev spaces  $\mathcal{K}_{a,p}^m(K)$, which for $m\in \mathbb{N}_0$, $1\leq p<\infty $,  and $a\in \mathbb{R}$  are  defined as the collection of all measurable functions such that 
	 \begin{align*}
	 \|u|\mathcal{K}_{a,p}^m(K)\|:=\left(
	 \displaystyle\sum_{|\alpha|\leq m}\int_{K} |\rho(x)^{|\alpha|-a}\partial^{\alpha}u(x)|^p\: dx
	 \right)^{1/p}<\infty,
	 \end{align*}
	 with  weight function $
	\rho(x):=\min(1,\mathrm{dist}(x,S))$ for $x\in K$. 
   In particular,  the  scales coincide if
   $$m=l,\: \delta=(\delta_1,...,\delta_n)=(l-a,...,l-a)\quad \text{and}\quad \beta=l-a.$$ To show this, one has  to prove that the power of the weight function of $\mathcal{K}_{a,p}^m(K)$ is equivalent to the power of the weight function of $ V_{\beta,\delta}^{l,p}(K)$ with these parameter assumptions, i.e.,
   $$\rho_{0}(x)^{\beta-l+|\alpha|}\prod_{k=1}^{n} \bigg( \dfrac{r_k(x)}{\rho_{0}(x)}  \bigg)^{\delta_k-l+|\alpha|} \sim \rho(x)^{|\alpha|-a}. $$
   This follows from \cite[p. 90, Subsection 3.1.1, (3.1.2)]{MR10}. 
   \item As an alternative approach to Definition \ref{Vspace} one could  define spaces $\tilde{V}_{\beta,\delta}^{l,p}(K)$, which contain all measurable functions such that the norm \eqref{norm} is finite, and investigate under which conditions the set $C^{\infty}_{\ast}(K,S)$ is dense in $\tilde{V}_{\beta,\delta}^{l,p}(K)$, i.e., when these spaces coincide with the spaces ${V}_{\beta,\delta}^{l,p}(K)$. This interesting question is postponed to a forthcoming paper.  
   \end{itemize}
 
\end{remark}

   Next we provide an equivalent norm in $ V_{\beta,\delta}^{l,p}(K) $. For this let 
  $\varphi_{j}$ be infinitely differentiable functions depending only on $\rho_0(x)=|x|$ such that 
   \begin{align}
   \label{zeta}
   \text{supp }\varphi_{j}\subset \{2^{-j-1}<|x|<2^{-j+1}\},\quad |\partial^{\alpha}\varphi_{j}(\cdot)|\leq c_{\alpha} 2^{|\alpha| j},\quad \displaystyle\sum_{j=j_0}^{\infty}\varphi_{j}=1,
   \end{align}
   where $c_{\alpha}$ are constants independent of $j$. 
   
   \begin{remark}
   	If $K$ is a bounded polyhedral cone, then w.l.o.g. we assume $j_0=0$. If it is unbounded, then $j_0=-\infty$ is required.
   \end{remark}

The following lemma is a direct consequence of the localization result in \cite[p.~91, Lemma 3.1.2]{MR10} and has a more convenient form for our setting in order to prove the existence of the extension operator.

\begin{lemma}
	\label{lemma2}
	Let $K$ be a polyhedral cone,  $l\in \mathbb{N}_0$, $ \beta\in \mathbb{R}$, $\delta\in\mathbb{R}^n $, and $1\leq p<\infty$. 
	Then for all $u\in V_{\beta,\delta}^{l,p}(K)$, 
	\begin{align*}
	\|u|V_{\beta,\delta}^{l,p}(K)\|^p \sim \displaystyle\sum_{j=j_0}^{\infty}\|\varphi_j u|V_{\beta,\delta}^{l,p}(K)\|^p,
	\end{align*}
	where $(\varphi_j)_{j\geq j_0}$ is  resolution of unity satisfying \eqref{zeta} as above.
\end{lemma}
\begin{remark}
Note that Lemma \ref{lemma2} actually holds for every family $(\varphi_j)_{j\geq j_0}$ as above, which satisfies the weaker assumption $\displaystyle\sum_{j=j_0}^{\infty}\varphi_j\sim 1$. 
\end{remark}

\section{An extension operator for the spaces $V_{\beta,\delta}^{l,p}(K)$}

In this section we construct an extension operator for the space $V_{\beta,\delta}^{l,p}(K)$. Our main result reads as follows:

 \begin{theorem}(Extension operator)
\label{extop}

Let $K\subset \mathbb{R}^3 $ be a  polyhedral cone and let $l\in \mathbb{N}_0$, $ \beta\in \mathbb{R}$, $\delta\in\mathbb{R}^n $ and $1\leq p<\infty$. Then there exists a bounded linear extension operator $$\mathfrak{E}:V_{\beta,\delta}^{l,p}(K,S)\rightarrow V_{\beta,\delta}^{l,p}(\mathbb{R}^3,S),$$ where $S$ is the singularity set of $K$.
\end{theorem}

\begin{remark}
The norm of the space $V_{\beta,\delta}^{l,p}(\mathbb{R}^3,S) $ is defined similarly as the norm of $ V_{\beta,\delta}^{l,p}(K,S)$ replacing the integral domain $K$ by $ \mathbb{R}^3$.
\end{remark}

\subsection{Extension operator for special Lipschitz domains}

 In order to prove Theorem \ref{extop} we need some preparations. We follow the proof of the extension theorem of Stein, who originally proved the theorem for classical Sobolev spaces. This was generalized by Hansen in \cite{Han15} to the setting of Kondratiev spaces. Since we now deal with mixed weights in the context of the spaces $V_{\beta,\delta}^{l,p}(K)$, we need to make some careful modifications. We start with an extension theorem for spaces defined on special Lipschitz domains, which  is a counterpart of the corresponding theorem for Kondratiev spaces  \cite[p. 582, Appendix]{Han15}. For this we  need to deal with the spaces  $V_{\delta}^{l,p}(D,\mathbb{R}^2_{\ast})$ from  \cite[p. 24]{MR10}, where 
 \[
 \mathbb{R}_*^2:=\{ x\in \mathbb{R}^3:\: x_{2}=x_{3}=0  \}
 \]
 denotes the $x_1$-axis. 
 
 \begin{definition}
 Let $l \in \mathbb{N}_0,\: \delta\in \mathbb{R}$, and $1\leq p<\infty$. Furthermore, let $D\subset \mathbb{R}^3$ be a special Lipschitz domain, i.e.,
\begin{align}
\label{speclipdom}
D=\{   x\in \mathbb{R}^3: \: x=(x',x_3),\: x'\in \mathbb{R}^{2},\: x_3>\omega(x') \},
\end{align}
for some Lipschitz-continuous function $\omega:\: \mathbb{R}^{2}\rightarrow \mathbb{R}$.
Assume $\mathbb{R}_*^2\subset \partial D $ and  let $r(x)$ denote the distance of $x$ to $\mathbb{R}_*^2$. Then the space $V_{\delta}^{l,p}(D,\mathbb{R}^2_{\ast})$  is defined as the closure of 
$$C_{\ast}^{\infty}(D,\mathbb{R}_*^2)=\{u\big|_{D}: \ u\in C_0^{\infty}(\real^3\setminus \mathbb{R}_*^2)\}$$
with respect to the norm 
\begin{align*}
      \|u|V_{\delta}^{l,p}(D,\mathbb{R}_*^2) \|=\left(  
    \displaystyle\int_D \sum_{|\alpha|\leq l}r^{p(\delta-l+|\alpha|)} |\partial^{\alpha}u(x)|^p \: dx
    \right)^{1/p}.
\end{align*}
 \end{definition}

\begin{remark}
More general, let $D$ be a domain of polyhedral type where for the edge $M$ we have $M\subset \partial D$. Then the space $V_{\delta}^{l,p}(D,M)$ can be defined in a similar manner. In this case the function $r$ denotes the distance to the edge $M$. Furthermore, the norm of the space $ V_{\delta}^{l,p}(\mathbb{R}^3,\mathbb{R}_*^2)$ is the counterpart of the norm of $V_{\delta}^{l,p}(D,\mathbb{R}_*^2)$ replacing the integral domain $D$ by $\mathbb{R}^3$.
\end{remark}

Now we can state and proof an extension theorem for the above spaces defined on special Lipschitz domain.

 \begin{theorem}
\label{prop}
  Let $l \in \mathbb{N}_0,\: \delta\in \real$, and $1\leq  p<\infty$. Moreover, let $D\subset\mathbb{R}^3$ be a special Lipschitz domain and assume $\mathbb{R}_*^2\subset \partial D $. Then there exists a universal linear and bounded  extension operator 
  $$\mathfrak{E}: \ V_{\delta}^{l,p}(D,\mathbb{R}_*^2)\rightarrow V_{\delta}^{l,p}(\mathbb{R}^3,\mathbb{R}_*^2). $$
\end{theorem}

\begin{remark}\hfill \\[0.2cm]
\label{remark3.6}
	\begin{minipage}{0.5\textwidth} 
	In particular, note that  Theorem \ref{prop} holds for wedges. In this case we  choose $\omega$ as the Lipschitz function
	$$ \omega(x_1,x_2)=|x_2|,$$
	and describe the wedge as a special Lipschitz domain according to  \eqref{speclipdom}, i.e.,
		\end{minipage}\hfill 
		\begin{minipage}{0.35\textwidth}
		\hfill \\
			\def\svgwidth{150pt}
			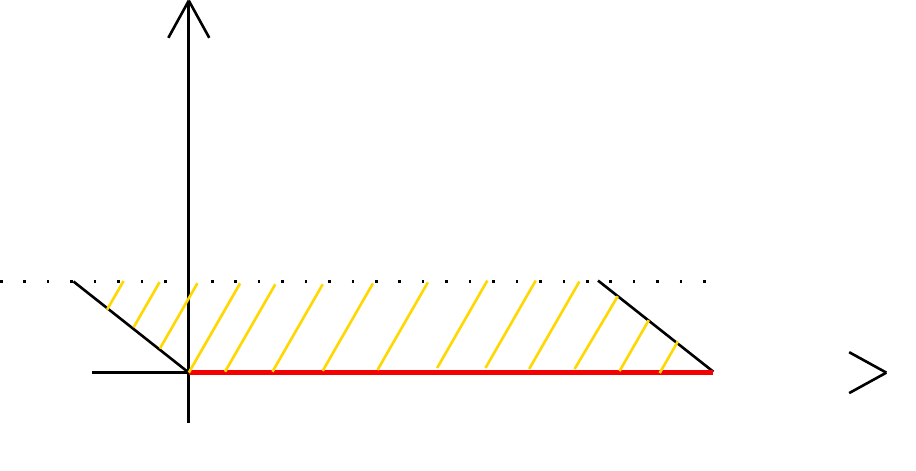
		\end{minipage} 
$$ W=\{  x\in \mathbb{R}^3: \: x=(x_1,x_2,x_3),\: x_3>\omega(x_1,x_2)   \}.  $$ 
	Thus, the weight in the $V_{\delta}^{l,p}(W,\mathbb{R}_*^2) $-norm  measures the distance of a point to  the $x_1$-axis,  which  represents the edge (and therefore the singularity set) of the wedge.
\end{remark}


\begin{proof}
For  the proof we rely on  calculations of Pieperbeck \cite{Piep} and  Hansen \cite{Han15}. 
Since by definition the set $C_{\ast}^{\infty}(D,\mathbb{R}_*^2)$ is dense in $V_{\delta}^{l,p}(D,\mathbb{R}_*^2) $,  it suffices to prove the theorem for this dense subset.

\textit{Step 1 (Preparations):} We set $\delta(\xi)=\mathrm{dist}(\xi,\partial D)$ as the distance of the point $\xi\in \overline{D}^c$ to the boundary of $D$. Its regularized version is denoted by $\Delta(\xi)$, cf. Stein \cite[p. 171, Theorem 2]{Stein}. In particular, it holds
\begin{align}
C_1 \delta(\xi)\leq \Delta(\xi)\leq C_2 \delta(\xi),
\quad \text{and}\qquad 
\label{aug11sok}
\left|  \partial_{\xi}^{\alpha} \Delta(\xi)      \right|\leq B_{\alpha}(\delta(\xi))^{1-|\alpha|},\quad \xi\in \overline{D}^c,
\end{align}
for constants $C_1$, $C_2$ and $B_{\alpha}$ independent of $D$.  
	We consider the point $\xi^0=(x^0,\omega(x^0))\in \partial D$ and denote by 
$$\Gamma_{\xi^0}=\{  \xi=(x',x_3):\  x_3<\omega(x^0),\ |x_3-\omega(x^0)|>M|x'-x^0|   \} $$
the lower cone with vertex at the point $\xi^0$, where $M$ is the Lipschitz constant  of $\omega$ and $x'\in\mathbb{R}^2,\: x_3\in\mathbb{R} $. Then we clearly have $\Gamma_{\xi^0}\cap \overline{D}=\{ \xi^0\}$. Moreover, the point $p=(x^0,y)$, $y<\omega(x^0)$, lies on the axis of the cone.
We make some elementary geometric calculations using the notation from the figure below. One can check easily that $\delta(p)\geq h$, and we see that
$q=\left(\frac{\omega(x^0)-y+M|x^0|}{M},y\right)\in\partial\Gamma_{\xi^0}$ according to the definition of $\Gamma_{\xi^0}$.

	\begin{minipage}{0.5\textwidth}\vspace{-5cm}
Furthermore, we have 
\begin{align*}
    a &=\omega(x^0)-y,\\
    b &= \frac{\omega(x^0)-y}{M},\\
    c &= \left( \left(  \frac{\omega(x^0)-y}{M}  \right)^2+(\omega(x^0)-y)^2    \right)^{1/2},
\end{align*}
and hence, 
$ 
    h=\frac{ab}{c}=\frac{\omega(x^0)-y}{(M^2+1)^{1/2}}.
$ 
	\end{minipage}\hfill 
		\begin{minipage}[b]{0.45\textwidth}
			\def\svgwidth{200pt}
			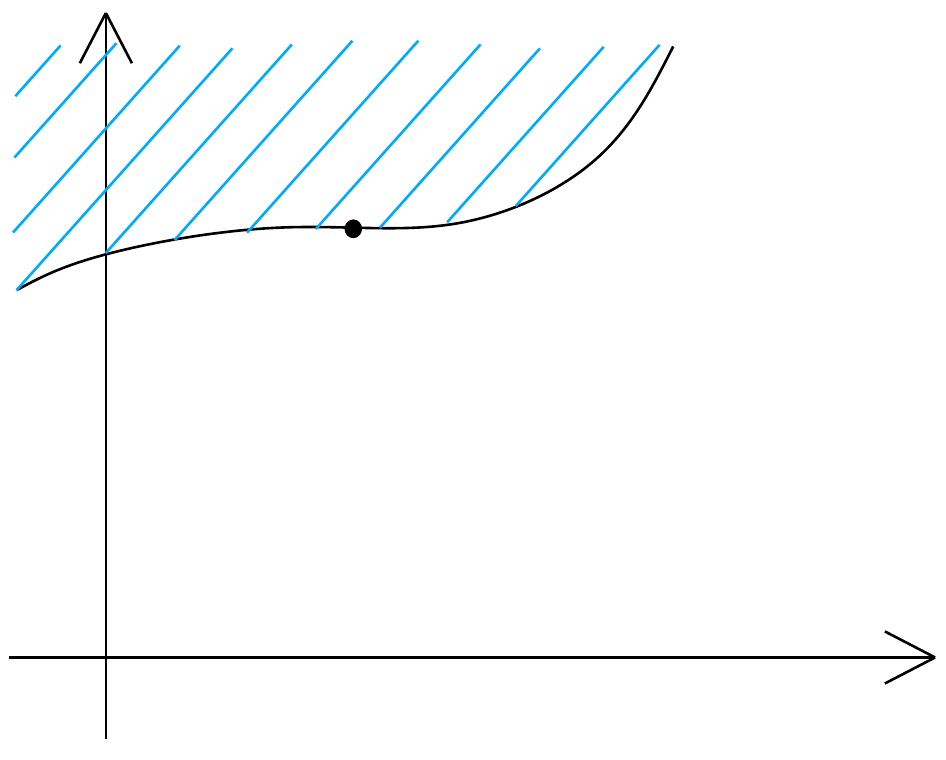
		\end{minipage} 


From these observations we obtain 
$ 
 \delta(x',x_3 )\geq  \frac{\omega(x')-x_3}{(1+M^2)^{1/2}},
$ 
for arbitrary $(x',x_3)\in \overline{D}^c$, which implies 
$$\omega(x')-x_3\leq C_3 \Delta(\xi).$$
Now we put $\delta^*(\xi)=2C_3\Delta(\xi)$ and obtain the estimate
$\delta^*(\xi)\geq 2(\omega(x')-x_3). $ 
From the definition of $\delta$ and since $D$ is a  special Lipschitz domain  we obtain
\begin{align*}
\delta(x',x_3)\leq d\left( (x',x_3), (x',\omega(x'))  \right)=|x_3-\omega(x')|=\omega(x')-x_3
\end{align*}
for all points $(x',x_3)\in \overline{D}^c$.
It further follows for all $\lambda>1$ that 
\begin{align}
\label{plusjel}
    x_3+\lambda \delta^*(x',x_3)\geq x_3+\delta^*(x',x_3)\geq x_3+2(\omega(x')-x_3)=2\omega(x')-x_3.
\end{align}
Finally, we also have
\begin{align}
    \label{bunkojel}
    \delta^*(x',x_3)=2C_3\Delta(x',x_3)\leq 2C_3C_2\delta(x',x_3)\leq 2C_3C_2(\omega(x')-x_3) .
\end{align}

\textit{Step 2 (Stein's extension operator):}
Stein defined the operator $\mathfrak{E}$ for a special Lipschitz domain $D$ and a function $u=g|_{D}$ where $g\in C_0^{\infty}(\mathbb{R}^d)$ 
by
\begin{align}
\label{steinextop}
    \mathfrak{E}u(x',x_3)=\displaystyle \int_{1}^{\infty} u(x',x_3+\lambda\delta^*(x',x_3))\psi(\lambda)\: d\lambda,\quad x'\in \mathbb{R}^{2},\: x_3\in \mathbb{R}, \: x_3<\omega(x'),
\end{align}
where $\psi:\: [1,\infty) \rightarrow \mathbb{R}$ is a rapidly decaying smooth function with $\displaystyle\int_{1}^{\infty}\psi(\lambda)\: d\lambda=1$ and $\displaystyle\int_{1}^{\infty}\lambda^k\psi(\lambda)\: d\lambda=0$ for all $k\in \mathbb{N}$. Such a function $\psi$ exists indeed according to \cite[p. 182, Chapter VI, Lemma 1]{Stein}. 
Now we fix a point $(x^0,\omega(x^0))\in \partial D.$ The properties of the function $\psi$ imply that $|\psi(\lambda)|\leq A\lambda^{-2}$ for some constant $A$. Using this and the previous estimates for $\delta^*$ we obtain for $x_3<\omega(x^0)$,
\begin{align}
\label{aug11}
|\mathfrak{E}u(x^0,x_3)|&\leq A \int_{1}^{\infty}|u(x^0,x_3+\lambda \delta^*(x^0,x_3))   | \frac{d\lambda}{\lambda^2}\notag\\
&=A\int_{x_3+ \delta^*(x^0,x_3)}^{\infty} |u(x^0,s)|\frac{\delta^*(x^0,x_3)^{-1}\: ds}{\delta^*(x^0,x_3)^{-2}(s-x_3)^2}\notag \\
&	\leq A\delta^*(x^0,x_3) \int_{x_3+ \delta^*(x^0,x_3)}^{\infty} |u(x^0,s)|(s-x_3)^{-2} \:ds\notag\\
&\leq A\delta^*(x^0,x_3) \int_{2\omega(x^0)-x_3}^{\infty} |u(x^0,s)|(s-x_3)^{-2}\: ds\notag\\
&\lesssim (\omega(x^0)-x_3)\int_{2\omega(x^0)-x_3}^{\infty} |u(x^0,s)|(s-x_3)^{-2} \:ds,
\end{align}
where we used the integral substitution $s:= x_3+\lambda\delta^*(x^0,x_3)$ with $ds=\delta^* \: d\lambda$ as well as formulas \eqref{plusjel} and \eqref{bunkojel} in the last two lines.\\
This pointwise estimate is now the basis for proving the relevant estimates in weighted Sobolev spaces.\\

\textit{Step 3 (Hardy's inequality and an estimate for $\alpha=0$):} The weight function of the space $V_{\delta}^{l,p}(D,\mathbb{R}_*^2) $ is given by
\begin{align}
\label{weightcone11}
 \rho(x):=r(x)^{\delta-l+|\alpha|},
\end{align}
where $r(x)$ denotes the distance of a point $x$ to $\mathbb{R}_*^2$ (i.e. the $x_1$-axis), thus $r(x)=x_2^2+x_3^2$. In this step we focus on the case $\alpha=0$, and also additionally assume $\delta\geq l$.

Now we multiply both sides of \eqref{aug11} with $\rho(x^0,x_3)$, 
take the $p$-th power on both sides and afterwards integrate w.r.t. $ x_3<\omega(x^0) $. This leads to 
\begin{align*}
&\displaystyle\int_{-\infty}^{\omega(x^0)}
\rho(x^0,x_3)^{p} |\mathfrak{E}u(x^0,x_3)|^p\: dx_3\\ &\lesssim \int_{-\infty}^{\omega(x^0)} (\omega(x^0)-x_3)^p \rho(x^0,x_3)^{p}
\left(  
\int_{2\omega(x^0)-x_3}^{\infty}|u(x^0,s)|(s-x_3)^{-2}\: ds
\right)^p \: dx_3.
\end{align*}
We make an integral substitution $ \widetilde{x_3}=\omega(x^0)-x_3 $ with $d\widetilde{x_3}=-dx_3$ and $\widetilde{s}=s-\omega(x^0)$ with $d\widetilde{s}=ds$. Then we obtain
\begin{align}
\label{aug5ca}
&\int_{-\infty}^{\omega(x^0)} (\omega(x^0)-x_3)^p \rho(x^0,x_3)^{p}
\left(  
\int_{2\omega(x^0)-x_3}^{\infty}|u(x^0,s)|(s-x_3)^{-2}\: ds
\right)^p \: dx_3\notag\\ &= \int_{0}^{\infty} \widetilde{x_3}^p \rho(x^0,\omega(x^0)-\widetilde{x_3})^p  \left(
\int_{\omega(x^0)+\widetilde{x_3}}^{\infty}|u(x^0,s)| (s-\omega(x^0)+\widetilde{x_3})^{-2}  \: ds
\right)^p \: d\widetilde{x_3}\notag\\
& \leq \int_{0}^{\infty} \widetilde{x_3}^p \rho(x^0,\omega(x^0)-\widetilde{x_3})^p  \left(
\int_{\widetilde{x_3}}^{\infty}
|u(x^0,\widetilde{s}+\omega(x^0))| \widetilde{s}^{-2 }\: d\widetilde{s}
\right)^p \: d\widetilde{x_3}.
\end{align}
We now intend to apply the following version of Hardy's inequality from \cite[p. 21]{Piep}:
Let $v,w:\mathbb{R}\rightarrow\mathbb{R}$ be measurable functions, $f\in L_p^+(\mathbb{R})=\{ g\in L_p(\mathbb{R}): \: g \: \text{is non-negative}   \}$, $1\leq p<\infty$ and $\frac{1}{p}+\frac{1}{q}=1$. Then there exists some finite constant $C(v,w)$ such that
\begin{align}
\label{hardya}
\left(\int_{0}^{\infty}\left|v(x)\int_{x}^{\infty}f(t)\: dt\right|^p\: dx   \right)^{1/p}\leq C(v,w) \left(   \int_{0}^{\infty} |w(x)f(x)|^p\: dx \right)^{1/p},
\end{align}
if it holds that
\begin{align}
\label{conditionhardya}
B(v,w):=\displaystyle\sup_{a>0} \left(\int_{0}^{a} |v(x)|^p \: dx   \right)^{1/p} \left(\int_{a}^{\infty} |w(x)|^{-q}\: dx  \right)^{1/q}<\infty.
\end{align}

We thus need to check the condition \eqref{conditionhardya}, where we choose the functions $v$ and $w$ as
\begin{align*}
v(x) & =x r(x^0,\omega(x^0)-x)^{\gamma},\\
w(x) & =x^2 r(x^0,x+\omega(x^0))^{\gamma},
\end{align*}
for some parameter $\gamma>0$ specified below.
The assumption $\real^2_*\subset\partial D$ entails $\omega(x_1,0)=0$ for all $x_1\in\real$, hence we find
\begin{align*}
    |\omega(x^0)|=|\omega(x_1^0,x_2^0)-\omega(x_1^0,0)|\leq M |(x_1^0,x_2^0)-(x_1^0,0)|=M|x_2^0|,
\end{align*}
which implies
\begin{align}
    \label{nov6}
  r(x^0,\omega(x^0)\pm x)^2 & \cong |x^0_2|^2+|\omega(x^0)\pm x  |^2  
   \leq |x^0_2|^2+ ( M|x_2^0|+x )^2
   \lesssim \max(|x_2^0|,x)^2.
\end{align}
Furthermore, it holds
\begin{align}
    \label{nov6b}
    |x_2^0|^2 \leq  r(x^0,\omega(x^0)\pm x)^2.
\end{align}

Using \eqref{nov6} we estimate the integral
\begin{align}
\label{estimua}
\int_{0}^{a} |v(x)|^p \: dx &=\int_{0}^{a} x^p r(x^0,\omega(x^0)-x)^{\gamma p} \: dx\notag\\
&\lesssim \max(|x_2^0|,a)^{\gamma p} \int_{0}^{a} x^p \: dx
\lesssim  \max(|x_2^0|,a)^{\gamma p} a^{p+1}.
\end{align}
Now we treat the second integral in \eqref{conditionhardya}. We consider two cases in order to obtain the needed estimate. \\
\textit{Case 1:} For $a<(2M+1)|x^0_2|$ we conclude with \eqref{nov6b}, that 
\begin{align*}
   \int_{a}^{\infty} |w(x)|^{-q} \: dx & = \int_{a}^{\infty} x^{-2q}r(x^0,x+\omega(x^0))^{-\gamma q} \: dx
    \lesssim |x_2^0|^{-\gamma q} a^{-2q+1}.
\end{align*}
\textit{Case 2:} For $a\geq (2M+1)|x^0_2|$ we have
$\  
    x\geq a \geq 2M|x^0_2| \geq 2 |\omega(x^0)|,\ 
$ 
from which we conclude
\begin{align*}
    r(x^0,\omega(x^0)+x)\geq |\omega(x^0)+x| \geq |x-|\omega(x^0)||\geq \frac{x}{2},
\end{align*}
and hence, 
\begin{align*}
     \int_{a}^{\infty} |w(x)|^{-q} \: dx &  \leq \int_{a}^{\infty} x^{-2q} \left( \frac{x}{2}   \right)^{-\gamma q} \: dx 
      \lesssim \int_{a}^{\infty} x^{-(2+\gamma)q}  \: dx
      \lesssim a^{-(2+\gamma)q+1}. 
\end{align*}
From the two cases above we now  see that 
\begin{align}
    \label{nov17}
      \int_{a}^{\infty} |w(x)|^{-q} \: dx  \lesssim a^{-2q+1}\max(a,|x_2^0|)^{-\gamma q}.
\end{align}
Now we estimate $B(v,w)$ using \eqref{estimua}, \eqref{nov17} and the fact that $\frac{1}{p}+\frac{1}{q}=1$. This gives
\begin{align*}
B(v,w)&\lesssim \displaystyle\sup_{a>0} \left(\max(|x_2^0|,a)^{\gamma} a^{\frac{p+1}{p}} \max(|x_2^0|,a)^{-\gamma} a^{\frac{-2q+1}{q}} \right)
=\sup_{a>0} 1=1<\infty,
\end{align*}
independent of $x^0$, and therefore,  condition \eqref{conditionhardya} is satisfied for our choice of functions $v$ and $w$. Applying now Hardy's inequality \eqref{hardya} with $f(x)=x^{-2}|u(x^0,\omega(x^0)+x)|$ and weights $v$ and $w$ for $\gamma=\delta-l>0$ to \eqref{aug5ca} we obtain

\begin{align*}
&\int_{0}^{\infty} \widetilde{x_3}^p \rho(x^0,\omega(x^0)-\widetilde{x_3})^p  \left(
\int_{\widetilde{x_3}}^{\infty}
|u(x^0,\widetilde{s}+\omega(x^0))| \widetilde{s}^{-2 }\: d\widetilde{s}
\right)^p \: d\widetilde{x_3}\\
&\lesssim \int_{0}^{\infty} \rho(x^0,\widetilde{x_3}+\omega(x^0))^{p}
|u(x^0,\widetilde{x_3}+\omega(x^0))|^p  \: d\widetilde{x_3}
= \int_{\omega(x^0)}^{\infty} \rho(x^0,x_3)^p |u(x^0,x_3)|^p \: dx_3.
\end{align*}
In conclusion, we obtain the inequality
\begin{align}
\label{aug5a}
\displaystyle\int_{-\infty}^{\omega(x^0)}
\rho(x^0,x_3)^{p} |\mathfrak{E}u(x^0,x_3)|^p\: dx_3
\lesssim  \int_{\omega(x^0)}^{\infty} \rho(x^0,x_3)^p |u(x^0,x_3)|^p \: dx_3.
\end{align}
Since $\mathfrak{E}u=u$ for $x_3>\omega(x^0)$ we trivially have
\begin{align}
\label{aug5ba}
\displaystyle\int_{\omega(x^0)}^{\infty}
\rho(x^0,x_3)^{p} |\mathfrak{E}u(x^0,x_3)|^p\: dx_3=\int_{\omega(x^0)}^{\infty}
\rho(x^0,x_3)^{p} |u(x^0,x_3)|^p\: dx_3.
\end{align}
Summing \eqref{aug5a} and \eqref{aug5ba} and integrating w.r.t. $x^0\in \mathbb{R}^{2}$ we obtain
\begin{align}
\int_{\mathbb{R}^{2}}\int_{\mathbb{R}} \rho(x^0,x_3) |\mathfrak{E}u(x^0,x_3)|^p \:dx_3\: dx^0
\lesssim \int_{\mathbb{R}^{2}} \int_{\omega(x^0)}^{\infty}\rho(x^0,x_3) |u(x^0,x_3)|^p \:dx_3\: dx^0. \label{aug5cba}
\end{align}
Taking the $1/p$-th power on both sides and applying Fubini's theorem yields
\begin{align*}
\|\rho\,\mathfrak{E}u|L_p(\real^3)\| \lesssim \|\rho u|L_p(\real^3)\|, 
\end{align*}
which takes care of the term with $\alpha=0$ in the norm-estimate for $\mathfrak{E}u$.\\

{\em Step 4 (The case $\alpha\neq 0$): } 
Next we need to investigate the partial derivatives of $\mathfrak{E}u$, i.e. we consider now $\partial^{\alpha}(\mathfrak{E}u)$, where $0<|\alpha|\leq l:$ Since $u$ is a test function of $C_0^{\infty}(\real^3\setminus \mathbb{R}_*^2)$ and, according to its definition, $\psi$ is a rapidly decaying function (i.e. decays faster than every polynomial), applying the Dominated Convergence Theorem we obtain that all of the derivatives of $u(\cdot)\psi(\cdot)$ are bounded, and hence we can change the order of integration and derivation as follows:
\begin{align*}
\partial^{\alpha}(\mathfrak{E}u)(x^0,x_3)& =\partial^{\alpha}\left(\displaystyle \int_{1}^{\infty} u(x^0,x_3+\lambda\delta^*(x^0,x_3))\psi(\lambda)\: d\lambda\right)\\
&=\int_{1}^{\infty}\psi(\lambda)\partial^{\alpha}\bigl[u(x^0,x_3+\lambda\delta^*(x^0,x_3))\bigr]\: d\lambda.
\end{align*}
We consider the first-order derivatives of $\mathfrak{E}u$ using the chain rule. This yields 

\begin{align}
\label{aug11_1}
&\frac{\partial(\mathfrak{E}u)}{\partial  x_i} (x^0,x_3)  \notag\\
&\quad =\displaystyle \int_{1}^{\infty} \psi(\lambda)\partial_{x_i}\bigl[ u(x^0,x_3+\lambda \delta^*(x^0,x_3))   \bigr] \:d\lambda\notag\\
&\quad =\displaystyle \int_{1}^{\infty} \psi(\lambda) 
\Bigl(\lambda\delta_{x_i}^*(x^0,x_3) u_{x_3}(x^0,x_3+\lambda\delta^*(x^0,x_3))+u_{x_i}(x^0,x_3+\lambda\delta^*(x^0,x_3)     )\Bigr)
 \:d\lambda
\end{align}
for $1\leq i\leq 2$ and

\begin{align}
\label{aug11_2}
\frac{\partial(\mathfrak{E}u)}{\partial x_3} (x^0,x_3)  &=\displaystyle \int_{1}^{\infty} \psi(\lambda)\partial_{x_3}\bigl[ u(x^0,x_3+\lambda \delta^*(x^0,x_3))   \bigr] \:d\lambda\notag\\
&= \displaystyle \int_{1}^{\infty} \psi(\lambda) 
\bigl(\lambda\delta_{x_3}^*(x^0,x_3)+1\bigr) u_{x_3}(x^0,x_3+\lambda\delta^*(x^0,x_3))
\:d\lambda.
\end{align}
Clearly higher-order derivative can be calculated similar to \eqref{aug11_1} and \eqref{aug11_2}. Exemplary, we find, once again using the chain rule,

\begin{align}
\label{aug11megtobb}
\frac{\partial^2(\mathfrak{E}u)}{\partial x_3^2} (x^0,x_3)&=\frac{\partial}{\partial x_3}\left(\frac{\partial(\mathfrak{E}u)}{\partial x_3} (x^0,x_3)   \right)\notag\\
&= \frac{\partial}{\partial x_3} \left( \displaystyle \int_{1}^{\infty} \psi(\lambda) 
((\lambda\delta_{x_3}^*(x^0,x_3)+1) u_{x_3}(x^0,x_3+\lambda\delta^*(x^0,x_3))     )
\:d\lambda    \right)\notag\\
&=\displaystyle \int_{1}^{\infty} \psi(\lambda) 
\frac{\partial}{\partial x_3}\bigg((\lambda\delta_{x_3}^*(x^0,x_3)+1) u_{x_3}(x^0,x_3+\lambda\delta^*(x^0,x_3))     \bigg)
\:d\lambda\notag\\
&=\displaystyle \int_{1}^{\infty} \psi(\lambda) 
((\lambda\delta_{x_3}^*(x^0,x_3)+1)^2 u_{x_3^2} (x^0,x_3+\lambda\delta^*(x^0,x_3)) \notag\\&\qquad +\lambda\delta_{x_3^2}^*(x^0,x_3)u_{x_3}     (x^0,x_3+\lambda\delta^*(x^0,x_3))         )\: d\lambda\notag\\
&= \underbrace{\displaystyle \int_{1}^{\infty} \psi(\lambda) \lambda^2 
\delta_{x_3}^*(x^0,x_3)^2u_{x_3^2}(x^0,x_3+\lambda\delta^*(x^0,x_3))\:d\lambda}_{=:I}\notag\\
&\qquad + \underbrace{ 
2\int_{1}^{\infty} \psi(\lambda)  \lambda \delta_{x_3}^*(x^0,x_3)u_{x_3^2}(x^0,x_3+\lambda\delta^*(x^0,x_3))\:d\lambda}_{=:II}
         \notag\\
&\qquad + \underbrace{   
\displaystyle \int_{1}^{\infty} \psi(\lambda) u_{x_3^2}(x^0,x_3+\lambda\delta^*(x^0,x_3))\:d\lambda}_{=:III}
       \notag\\
&\qquad +\underbrace{      \displaystyle \int_{1}^{\infty}\psi(\lambda)\lambda  \delta_{x_3^2}^*(x^0,x_3)  u_{x_3}(x^0,x_3+\lambda\delta^*(x^0,x_3))\:d\lambda }_{=:IV}.
\end{align}
Derivatives of higher order just result in a higher number of terms of similar form. To derive norm-estimates, every term will be treated separately, using the moment conditions for $\psi$ together with suitable Taylor expansions. This will be illustrated again by discussing  the case
$\frac{\partial^2(\mathfrak{E}u)}{\partial x_3^2}$ in detail. Thus we need to derive estimates for the integrals I--IV in \eqref{aug11megtobb}. We first recall that according to \eqref{aug11sok} we have 
\begin{align}
\label{aug}
\partial^{\alpha}\delta^*\leq c_{\alpha}(\delta^*)^{1-|\alpha|}
\end{align}
and
$  |\psi(\lambda)|\leq A_k\lambda^{-k} \quad \text{for any}\quad k\in \mathbb{N}_0.$ 
We consider $ I $ for $x_3<\omega(x_0)$ and obtain with  a similar argument as in the case $\alpha=0$ [cf. \eqref{aug11}]
\begin{align*}
|I|&\leq \displaystyle \int_{1}^{\infty} A_4\lambda^{-4} \lambda^2 c_{x_3}^2| u_{x_3^2}(x^0,x_3+\lambda\delta^*(x^0,x_3))  |  \: d\lambda \\
& \lesssim  A_4c_{x_3}^2 (\omega(x^0)-x_3)  \int_{2\omega(x^0)-x_3}^{\infty}  |u_{x_3^2}(x^0,s)   | (s-x_3)^{-2}\: ds.
\end{align*}
For $II$ we similarly find
\begin{align*}
|II|& \leq
2A_3 \int_{1}^{\infty} \lambda^{-3} \lambda c_{x_3} | u_{x_3^2}(x^0,x_3+\lambda\delta^*(x^0,x_3))  |
 \: d\lambda \\
& \lesssim  A_3c_{x_3} (\omega(x^0)-x_3)  \int_{2\omega(x^0)-x_3}^{\infty}  |u_{x_3^2}(x^0,s)   | (s-x_3)^{-2}\: ds.
\end{align*}
Moreover, concerning $III$ we get [cf. \eqref{aug11}], 
\begin{align*}
|III|& \leq
A_2 \int_{1}^{\infty} \lambda^{-2} | u_{x_3^2}(x^0,x_3+\lambda\delta^*(x^0,x_3))  |
\: d\lambda \\
& \lesssim  A_2 (\omega(x^0)-x_3)  \int_{2\omega(x^0)-x_3}^{\infty}  |u_{x_3^2}(x^0,s)   | (s-x_3)^{-2}\: ds.
\end{align*}
The last term $IV$ needs a bit more care, because it contains the term $\delta_{x_3^2}^*$ and for that the estimate \eqref{aug} alone is not enough. Instead we use a Taylor expansion to rewrite $u_{x_3}$ as
\begin{align*}
u_{x_3}(x^0,x_3+\lambda\delta^*(x^0,x_3))=u_{x_3}(x^0,x_3+\delta^*(x^0,x_3))+ \displaystyle \int_{x_3+\delta^*(x^0,x_3)}^{x_3+\lambda\delta^*(x^0,x_3)} 
u_{x_3^2}(x^0,t)\: dt.
\end{align*}
Due to the properties of $\psi$, in particular $ \displaystyle \int_{1}^{\infty}\lambda \psi(\lambda) \: d\lambda=0$, we calculate for the term $IV$, 
\begin{align*}
&\displaystyle \int_{1}^{\infty}\psi(\lambda)\lambda  \delta_{x_3^2}^*(x^0,x_3)  u_{x_3}(x^0,x_3+\lambda\delta^*(x^0,x_3))\:d\lambda \\
& =\displaystyle \int_{1}^{\infty}\psi(\lambda)\lambda  \delta_{x_3^2}^*(x^0,x_3)  u_{x_3}(x^0,x_3+\delta^*(x^0,x_3))\:d\lambda \\
& \qquad + \displaystyle \int_{1}^{\infty}\psi(\lambda)\lambda  \delta_{x_3^2}^*(x^0,x_3)  \int_{x_3+\delta^*(x^0,x_3)}^{x_3+\lambda\delta^*(x^0,x_3)} 
u_{x_3^2}(x^0,t)\: dt\: d\lambda\\
& =\delta_{x_3^2}^*(x^0,x_3)  u_{x_3}(x^0,x_3+\delta^*(x^0,x_3))\displaystyle \int_{1}^{\infty}\psi(\lambda)\lambda  \:d\lambda \\
& \qquad + \displaystyle \int_{1}^{\infty}\psi(\lambda)\lambda  \delta_{x_3^2}^*(x^0,x_3)  \int_{x_3+\delta^*(x^0,x_3)}^{x_3+\lambda\delta^*(x^0,x_3)} 
u_{x_3^2}(x^0,t)\: dt\: d\lambda\\
&=\int_{1}^{\infty}\psi(\lambda)\lambda  \delta_{x_3^2}^*(x^0,x_3)  \int_{x_3+\delta^*(x^0,x_3)}^{x_3+\lambda\delta^*(x^0,x_3)} 
u_{x_3^2}(x^0,t)\: dt\: d\lambda.
\end{align*}
This leads to the following estimate
\begin{align*}
|IV|&\lesssim (\delta^*(x^0,x_3))^{-1}A_4  \displaystyle \int_{1}^{\infty}
\left(\int_{x_3+\delta^*(x^0,x_3)}^{x_3+\lambda\delta^*(x^0,x_3)} 
|u_{x_3^2}(x^0,t)|\: dt     \right)\lambda^{-3} \: d\lambda\\
&=(\delta^*(x^0,x_3))^{-1}A_4 \int_{x_3+\delta^*(x^0,x_3)}^{\infty} \left|
\left(  
\int_{\frac{t-x_3}{\delta^*(x^0,x_3)}}^{\infty}\lambda^{-3} \: d\lambda
  \right)u_{x_3^2}(x^0,t)
\right|\: dt\\
&\sim (\delta^*(x^0,x_3))^{-1} \int_{x_3+\delta^*(x^0,x_3)}^{\infty}(\delta^*(x^0,x_3))^2
| u_{x_3^2}(x^0,t)  | \frac{dt}{(t-x_3)^2}\\
&\lesssim (\omega(x^0)-x_3)  \int_{2\omega(x^0)-x_3}^{\infty}  |u_{x_3^2}(x^0,s)   | (s-x_3)^{-2}\: ds\,.
\end{align*}
Altogether,  the integral estimates for $I$-$IV$ imply 
$$\left|\frac{\partial^2(\mathfrak{E}u)}{\partial x_3^2} (x^0,x_3)\right|\lesssim (\omega(x^0)-x_3)  \int_{2\omega(x^0)-x_3}^{\infty}  |u_{x_3^2}(x^0,s)   | (s-x_3)^{-2}\: ds, $$
which is the counterpart of the pointwise estimate \eqref{aug11}. From there on, the estimate
$$\Bigl\|\rho\,\frac{\partial^2(\mathfrak{E}u)}{\partial x_3^2}\Big|L_p(\real^3)\Bigr\|
    \lesssim\Bigl\|\rho\,\frac{\partial^2 u}{\partial x_3^2}\Big|L_p(\real^3)\Bigr\|$$
follows mutatis mutandis as in Step 3, simply by replacing $u$ by $u_{x_3^2}$ and with $\rho(x)=r(x)^{\delta-l+2}$. Since the higher order partial derivatives of $(\mathfrak{E}u)$ can be treated similarly as explained above, the theorem is proved for $\delta\geq l$.\\
{\em Step 5 (The case $\delta< 0$): } We  modify the estimate \eqref{aug11} in order to allow negative $\delta$ as follows:  According to the conditions of the function $\psi$ for  Stein's extension operator  we can replace $|\psi(\lambda)|\leq A\lambda^{-2}$ by $|\psi(\lambda)|\leq A_{\kappa}\lambda^{-\kappa}$, where   $\kappa\in \mathbb{N}$  can be chosen sufficiently large. This implies an analogous estimate for $\mathfrak{E}u(x^0,x_3)$ as in \eqref{aug11}, in particular,  as a substitute we obtain 
\begin{align}
|\mathfrak{E}u(x^0,x_3)|
&\lesssim (\omega(x^0)-x_3)^{\kappa-1}\int_{2\omega(x^0)-x_3}^{\infty} |u(x^0,s)|(s-x_3)^{-\kappa} \:ds. 
\end{align}
Then, when applying Hardy's inequality \eqref{hardya}, we now  choose the  functions 
$$v(x) =x^{\kappa-1} r(x^0,\omega(x^0)-x)^{\gamma}
\qquad \text{and}\qquad w(x) =x^{\kappa} r(x^0,x+\omega(x^0))^{\gamma},$$
where now $\gamma=\delta-l<0$. Similar calculations as in Step 3 above yield  that the  condition \eqref{conditionhardya} of  Hardy's inequality  holds if $(\kappa+\gamma)q>1$. Therefore, we can take any negative $\gamma<0$  if  $\kappa$ is  sufficiently large, which results in  the estimate \eqref{aug5cba}  with  $\rho(x)=r(x)^{\delta-l}$ and $\delta<0$ and proves the case $\alpha=0$. Concerning $\alpha\neq 0$, the arguments have to be adapted according to Step 4.  

\end{proof}


\subsection{Extension operator for a layer of the cone}

\begin{figure}[!ht] 
\begin{minipage}{0.65\textwidth}              
We now consider a fixed layer  $K'\subset K$ of the cone, i.e.,
 \begin{align}
 \label{nov19}
   K'=\{ x\in K: \: C_1<|x|<C_2  \}  
 \end{align}
for some constants $0<C_1<C_2<\infty$ and give an extension for functions belonging to $V_{\beta,\delta}^{l,p}$ from the layer $K'$  to $ \mathbb{R}^3$. In the proof we use the extension operator for special Lipschitz domains constructed in  Theorem \ref{prop}.
\end{minipage}\hfill \begin{minipage}{0.25\textwidth}       
    \def\svgwidth{100pt}    
    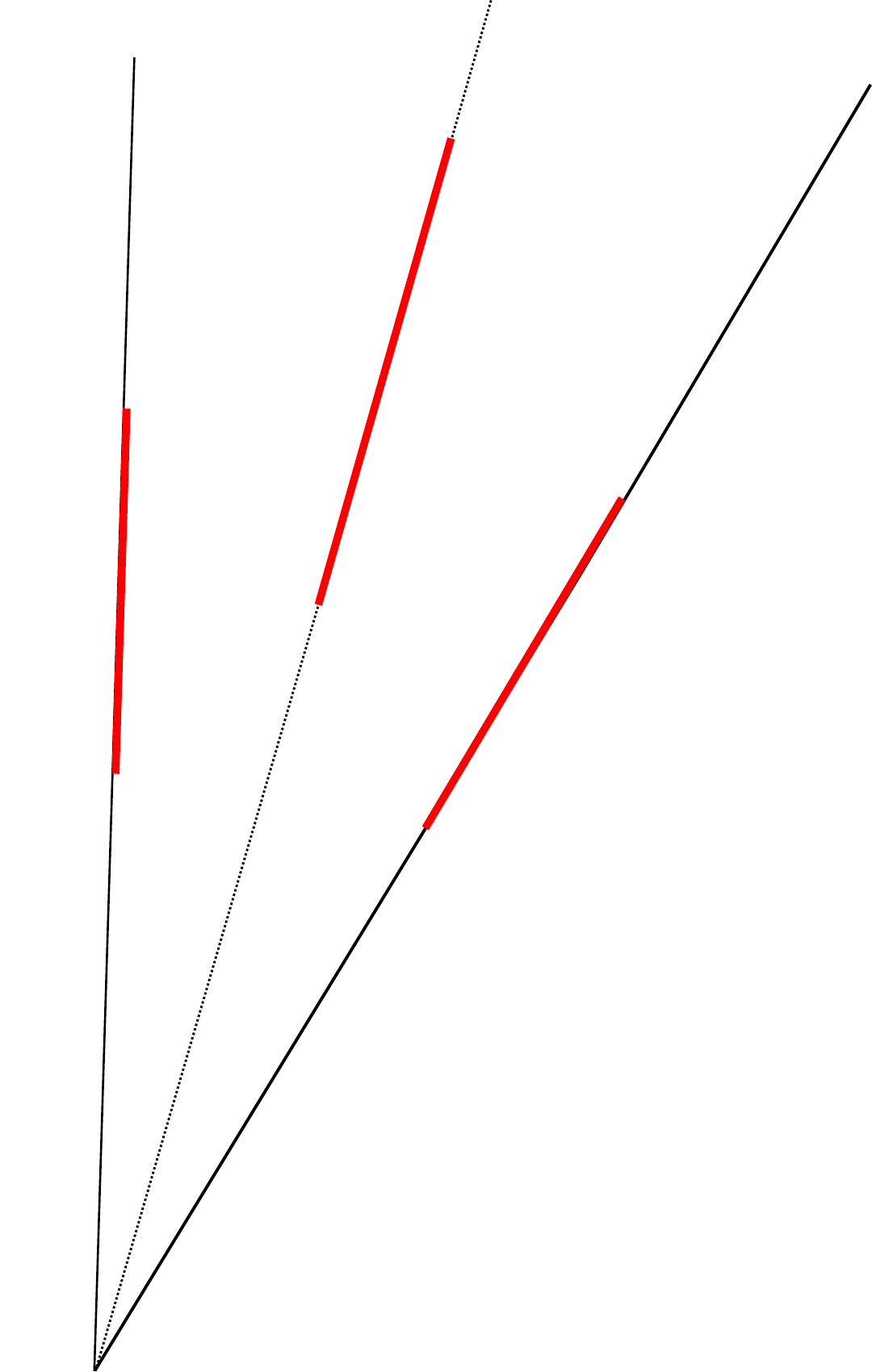  
    \end{minipage}\\ 
\end{figure}

\begin{remark}\label{notation-expl}  In the sequel we will  (in slight abuse) frequently use the notation $u\in V_{\beta,\delta}^{l,p}(K', S')$, where $K'$ denotes a fixed layer of the cone and  $S':= \overline{K'}\cap S$ the corresponding part of the singularity set, 
when dealing with functions $u\in V_{\beta,\delta}^{l,p}(K, S)$, where we only  consider their values on $K'$.  We remark that  in view of Lemma \ref{equiv-dist} below this is in good agreement with  the definition of the spaces $V_{\beta,\delta}^{l,p}(K,S)$, since the expression 
\begin{align}
\label{norm-2}
\|u| V_{\beta,\delta}^{l,p}(K',S')\|
&=\displaystyle \bigg( \int_{K'}\sum_{|\alpha|\leq l} 
\rho_{0}(x)^{p(\beta-l+|\alpha|)}\prod_{k=1}^{n} \bigg( \dfrac{r'_k(x)}{\rho_{0}(x)}  \bigg)^{p(\delta_k-l+|\alpha|)} |\partial^{\alpha}u(x)|^p\: dx
\bigg)^{1/p}<\infty, 
\end{align}
where $r'_{k}$ denotes the distance to the edge $M_k\cap \overline{K'}$ is equivalent with $\|u|V_{\beta,\delta}^{l,p}(K', S)\|$. Moreover, if $u$ has support in $K'$ we also see that the expressions $ \|u|V_{\beta,\delta}^{l,p}(K', S')\|$ and $ \|u|V_{\beta,\delta}^{l,p}(K, S)\|$ are equivalent. 
\end{remark}

We now construct an extension operator   from the layer $K'$ to  $\real^3$ in our weighted Sobolev spaces as follows.

\begin{lemma}
	\label{layerextopmod}
 Let $K\subset \mathbb{R}^3$ be a polyhedral cone and let $K'\subset K$ be a fixed layer as in \eqref{nov19}. Furthermore, let $l\in \mathbb{N}_0,\:  \delta \in \mathbb{R}^n $, $\beta\in \real$, and $1\leq   p< \infty$. Then there exists a bounded linear extension operator 
 $$\mathfrak{E}:\: V_{\beta,\delta}^{l,p}(K',S')\longrightarrow V_{\beta,\delta}^{l,p}(\mathbb{R}^3,S),$$ 
 where $S=M_1\cup ...\cup M_n$ and $S':= \overline{K'}\cap S$.
\end{lemma}
\begin{proof} 


\textit{Step 1:}  One can see easily that the layer $K'$ can be covered by finitely many diffeomorphic images of bounded wedges $W_1,...,W_n$. Therefore, we first give an extension for the single wedges $W_k$ w.r.t. the spaces $V_{\delta_k}^{l,p}$, where $\delta_k\in \mathbb{R}$. Then, using a localization argument we provide an extension operator for the layer $K'$ in the space $V_{\beta, \delta}^{l,p}$ with $\beta\in \real$, $\delta=(\delta_1,...,\delta_n)$ based on the operators for the single wedges.\\

           \begin{minipage}{0.65\textwidth}  
            Let $\Omega$ be a bounded wedge representing the neighbourhood of the edge $M_k$ on $K'$, which in cylinder coordinates is given by
$$\Omega=\{ x=(r,\varphi,x_3):\: 0<x_3<1,\  0<r<1,\ 0<\varphi<\varphi_0  \}$$
for some $0<\varphi_0<2\pi$.  
    \end{minipage}\hfill \begin{minipage}{0.25\textwidth}    
   \def\svgwidth{100pt}    
    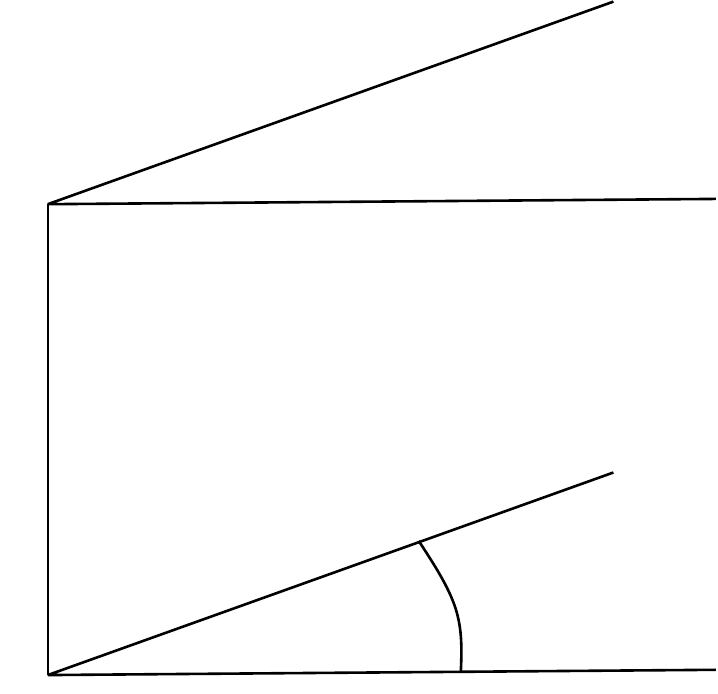  
  5  \label{fig:omega}
    \end{minipage}\\

In this case the weight function of the space $V_{\delta_k}^{l,p}$, 
\begin{align}
\label{weightcone2}
\rho(x):= r_k(x)^{\delta_k-l+|\alpha|}\sim r(x)^{\delta_k-l+|\alpha|},
\end{align}
is equivalent to the power of the distance of a point $x$ to the edge of the bounded wedge $\Omega$ (represented by the $x_3$-axis in cylinder coordinates).\\
    
 \begin{minipage}{0.65\textwidth}  
 Via multiplication with a cut-off function, we can reduce the problem to considering the half-unbounded wedge
$$\Omega_+=\{ x=(r,\varphi,x_3):\: x_3>0,\  r>0,\ 0<\varphi<\varphi_0  \}.$$
\end{minipage}\hfill \begin{minipage}{0.25\textwidth}    
    \def\svgwidth{100pt}    
    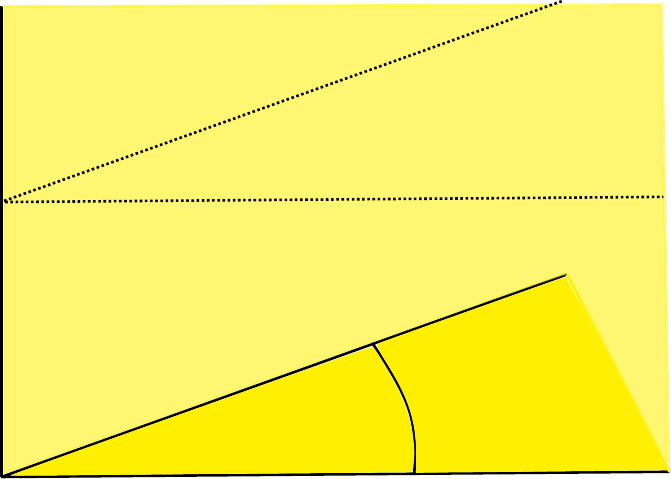  
     \end{minipage}\\[2\baselineskip]
More precisely, let $\eta_1$ be a smooth function on $\real^3$ such that $\eta(x)=1$ for $x_3\geq 2/3$ and $\eta(x)=0$ for $x_3\leq 1/3$, and put $\eta_2=1-\eta_1$.
Then we decompose $u=\eta_1 u+\eta_2 u$. Clearly, we can extend $\eta_1 u$ to $x_3<0$ and $\eta_2 u$ to $x_3>1$, respectively, by zero --- both situations then obviously being equivalent to discussing functions on $\Omega_+$.

\medskip

\begin{minipage}{0.65\textwidth}  
We proceed as follows.
First we give an extension from the half-unbounded wedge to the unbounded wedge,
$$\Omega_0=\{ x=(r,\varphi,x_3):\: x_3\in \mathbb{R},\  r>0,\ 0<\varphi<\varphi_0  \},$$
and subsequently refer to Theorem \ref{prop} for another extension to $\mathbb{R}^3$ for the special Lipschitz domain $\Omega_0$ (see Remark \ref{remark3.6}), where the $x_1$-axis is rotated to the edge of the wedge. \\
\end{minipage}\hfill \begin{minipage}{0.25\textwidth}    
	\def\svgwidth{100pt}    
	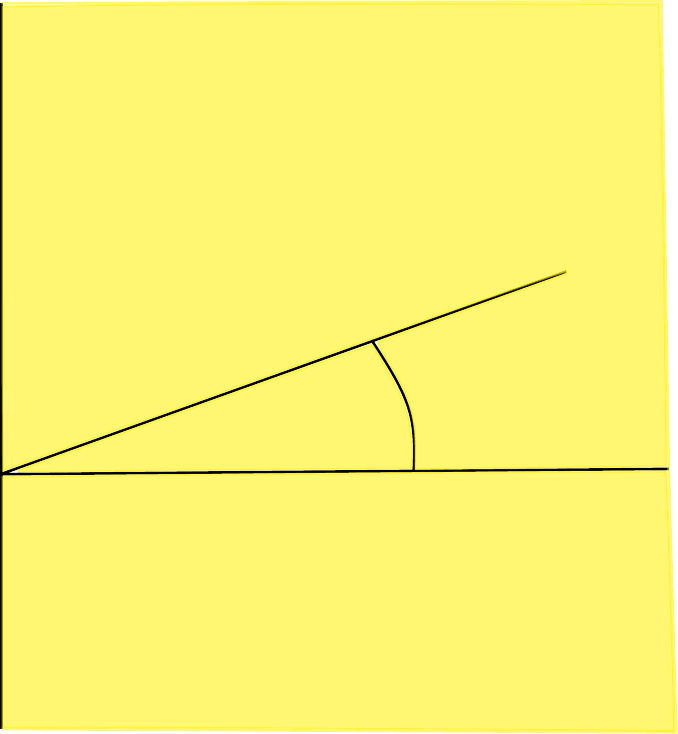  
 \end{minipage}\\

We summarize the problem in a diagram: 
\begin{align*}
\Omega  \ \xrightarrow[\text{function }]{\text{cut-off}} \ \Omega_+ \ 
 \xrightarrow[\text{  }]{\boxed{{\text{Steps } 2+3}}} \ \Omega_0 \ \xrightarrow[\text{  }]{\text{Theorem \ref{prop}}} \ \mathbb{R}^3
\end{align*}
Thus, we only need to show the second arrow and use Stein's extension operator in order to give this extension.\\

\textit{Step 2:} We need some more notation and preparatory remarks: 
Let $G$ be a special Lipschitz domain as in Theorem \ref{prop},
$$G=\{  x\in \mathbb{R}^3: \: x=(x',x_3), \ x'\in \mathbb{R}^2, \ x_3>\omega(x')  \},$$
where $\omega:\mathbb{R}^2\longrightarrow \mathbb{R}$ is a Lipschitz function.  
Now we consider the sets 
$$G_{x'}^-=\{ (x',x_3): \: x_3<\omega(x')  \}\qquad \text{and}\qquad  G_{x'}^+=\{ (x',x_3): \: x_3>\omega(x')  \}$$ for fixed $x'\in \mathbb{R}^2$. 
 \begin{figure}[ht] 
 \begin{minipage}{0.55\textwidth}
We see that Stein's extension operator from \eqref{steinextop} in points $x\in G_{x'}^-$ uses only function values of $u$ in $G_{x'}^+$. This allows us to apply the same definition also for functions given on the {\it restricted special Lipschitz domain}
$$ G=\{x\in \mathbb{R}^3:\:  x=(x',x_3),\ x'\in G_0, \ x_3>\omega(x')   \},$$
where $G_0\subset \mathbb{R}^2$ is a sufficiently smooth domain, which we can choose according to our needs and $\omega$ is a Lipschitz function on $G_0$. 
 \end{minipage}\hfill \begin{minipage}{0.4\textwidth}
      \centering              
    \def\svgwidth{200pt}    
    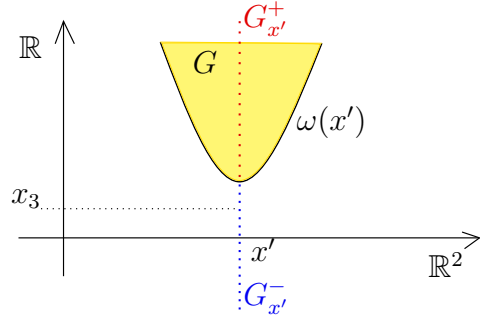  
    \caption{Special Lipschitz domain $G$ above graph of $\omega(x')$}
   \label{fig:grafikon}
   \end{minipage}\\ 
\end{figure}
 

This particularly applies to half-unbounded wedges by choosing
$ G_0=\{ (r,\varphi):\: r>0, 0<\varphi<\varphi_0  \},$ 
where $(r,\varphi)$ represents $x'\in \mathbb{R}^2$ in polar coordinates and the function $\omega$ can be defined as
 $$\omega(x')\equiv 0 \text{  for all  } x'\in G_0.$$
Then indeed $\Omega_+$ is a restricted special Lipschitz domain in view of
 $$G:=\Omega_+=G_0 \times \mathbb{R}^+=\{   
 (x',x_3):\: x'\in G_0,\  x_3>\omega(x')=0
 \}.$$
 
\textit{Step 3:} Now we can apply Theorem \ref{prop} on the domain $\Omega_+=G_0 \times \mathbb{R}^+$ with some careful modifications: 
For the norm estimates in the spaces $V_{\delta_k}^{l,p}$ we use Theorem \ref{prop}, in particular, the inequalities \eqref{aug5a} and \eqref{aug5ba} with $\omega(x')=0$ and weight $\rho(x',x_3)\equiv 1$  (note that the Hardy inequality \eqref{hardya} also works with $v=x$ and $w=x^2$ when $\rho\equiv 1$). This gives 
\begin{align}
\label{aug14a2}
\displaystyle\int_{-\infty}^{0}
 |\mathfrak{E}u(x',x_3)|^p\: dx_3
\lesssim  \int_{0}^{\infty} |u(x',x_3)|^p  \: dx_3.
\end{align}
and
\begin{align}
\label{aug14b2}
\displaystyle\int_{0}^{\infty}
 |\mathfrak{E}u(x',x_3)|^p\: dx_3=\int_{0}^{\infty}
 |u(x',x_3)|^p\: dx_3.
\end{align}
Then we  sum up  \eqref{aug14a2} and \eqref{aug14b2}, multiply with the weight $\rho(x',x_3)=|x'|^{\delta_k}$ (measuring the distance to the $x_3$-Axis) and integrate w.r.t. $x_0\in G_0$. Since the weight is independent of $x_3$ this leads to 
\begin{align}
\label{aug14c2}
\int_{G_0}\int_{\mathbb{R}} \rho(x',x_3) |\mathfrak{E}u(x',x_3)|^p \:dx_3\: dx'
\lesssim \int_{G_0} \int_{0}^{\infty}\rho(x',x_3) |u(x',x_3)|^p \:dx_3\: dx'.
\end{align}
Hence, we obtain 
\begin{align*}
\|\mathfrak{E}u|V_{\delta_k}^{0,p}(\Omega_0,\mathbb{R}^2_*)\| \lesssim \|u|V_{\delta_k}^{0,p}(\Omega_+,\mathbb{R}^2_*)\|.
\end{align*}
For $l>0$ we can use a similar argument as in Theorem \ref{prop}.  
Summarizing, Stein's extension operator can be modified to give an extension from $\Omega_+$ to $\Omega_0$ together with norm estimates in associated weighted Sobolev spaces $V_{\delta_k}^{l,p}$, $\delta_k\in \mathbb{R}$. Finally, on $\Omega_0$ we can apply the results of Theorem \ref{prop} to extend functions from $\Omega_0$ to $\mathbb{R}^3$.\\
We denote corresponding extension operators for the wedges $W_1,...,W_n$ by $\mathfrak{E}_1,...,\mathfrak{E}_n$.\\

\textit{Step 4:} The results for the single wedges can now be transferred to the layer $K'$, and spaces $V_{\beta,\delta}^{l,p} $, $\delta\in \real^n$,  where the weight function contains $r'_1,...,r'_n$, i.e., the distances to the edges of all wedges along the layer. In order to glue  these operators properly together we need some preparations: 
 Let 
 $$ \overline{K'}\subset \bigcup_{j=1}^n U_j, \qquad \text{where} \quad  W_j=K'\cap U_j \quad \text{is a wedge}$$ 
 as in Step 1 and $U_j$ does not intersect the edges $M_i$ for $i\neq j$. Since $K'$ is a bounded Lipschitz domain, the sets $U_j$ can be chosen such that there exists an $\varepsilon>0$ with $B(x,\varepsilon)\subset U_j$ for some $j$ for every $x\in \partial K'$. Then we define $U_j^{\varepsilon}=\{ x\in U_j:\: B(x,\varepsilon)\subset U_j  \}$. Moreover, let $\varphi_1,...,\varphi_n$ be non-negative smooth functions with $$\varphi_j(x)=1 \text{ on } U_j^{\varepsilon/2} \qquad \text{and}\qquad  \text{supp}\: \varphi_j\subset U_j.$$
 
 \begin{figure}[!h]
\begin{minipage}[b]{0.5\textwidth}\hspace{2cm}
  \def\svgwidth{180pt}
  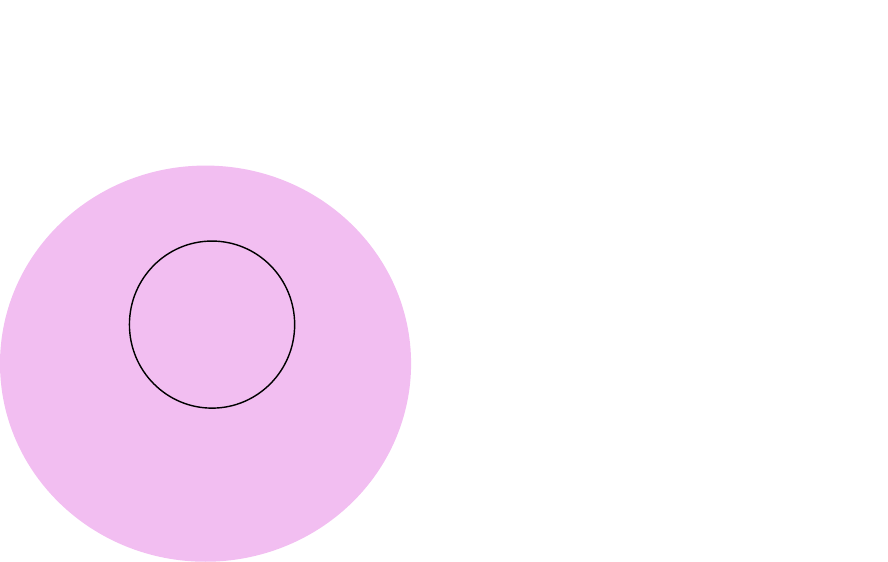
    \caption{The set $U_j$}
  \end{minipage}\hfill 
    \begin{minipage}[b]{0.5\textwidth} \hspace{0.8cm}
  \def\svgwidth{180pt}
 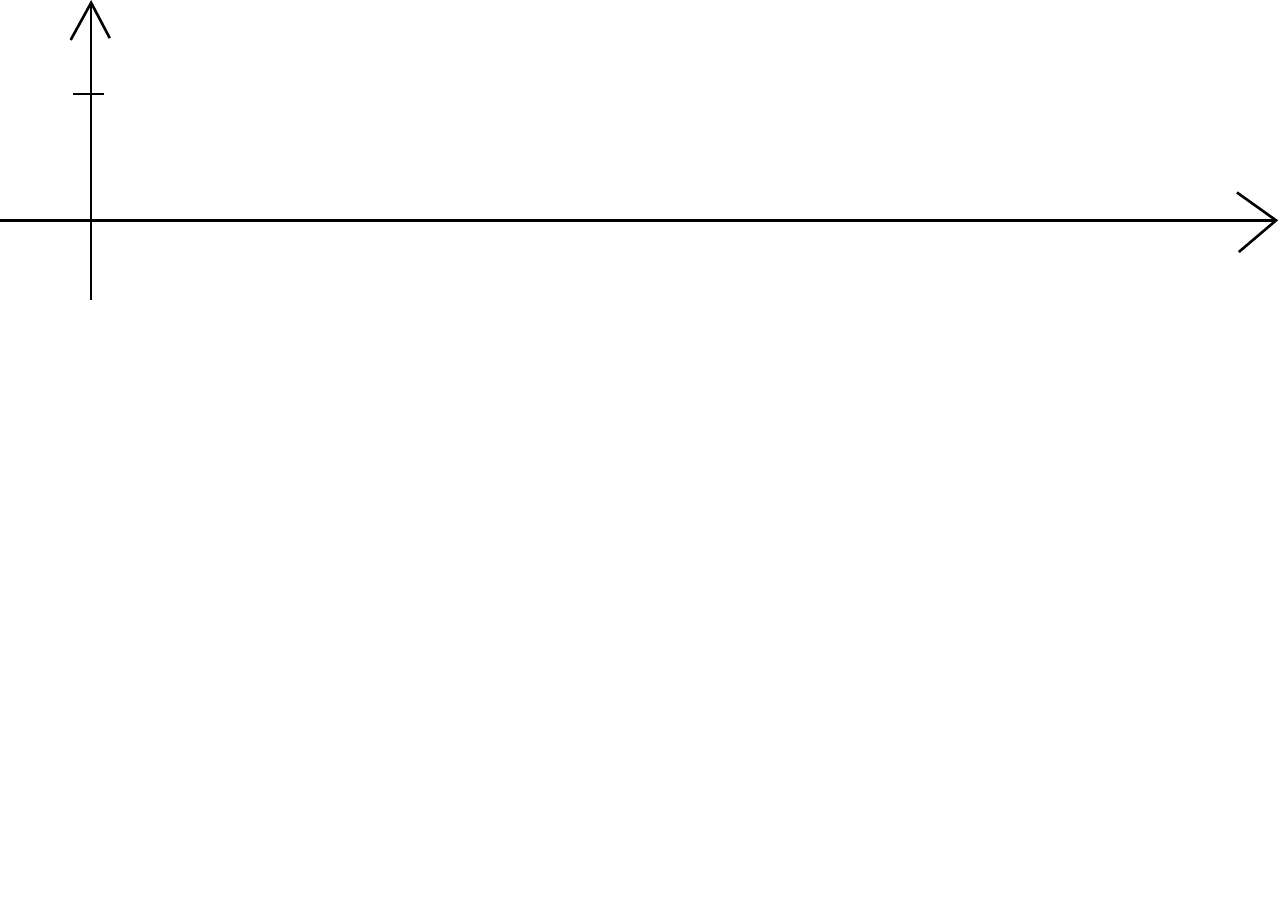
    \caption{The function $\varphi_j$}
  \end{minipage}
\end{figure}

 Furthermore, let $\Psi$ be a smooth function with 
$$\text{supp }\Psi \subset K'\qquad \text{and} \qquad \Psi(x)=1 \ \text{ for } \ x\in K' \ \text{ with }\  \mathrm{dist}(x,\partial K')>\varepsilon/2,$$
and let $\Phi$ be another smooth function with 
\begin{align*}
\Phi =1-\Psi \text{ on } K'\qquad \text{and} \qquad  \text{supp }\Phi\subset \{ x\in \mathbb{R}^3:\:
\mathrm{dist}(x,\partial K')<\varepsilon/2
  \} \subset \bigcup_j U_j^{\varepsilon/2}. 
\end{align*}
\begin{figure}[!h] 
\centering 
    \def\svgwidth{250pt}    
  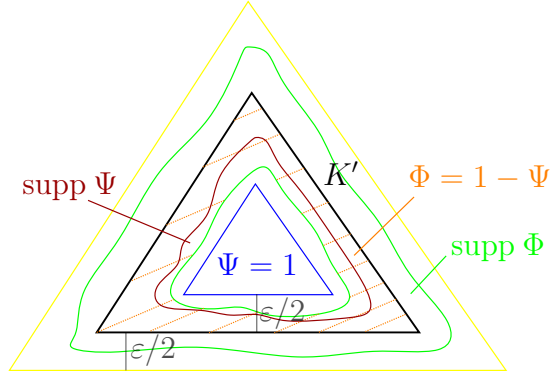  
   \caption{The functions $\Psi$ and $\Phi$ from above point of view}
  \label{fig:bmeg}
\end{figure}

We glue together the operators for the wedges and define 
$$\mathfrak{E}u(x)=\Psi(x)u(x)+\Phi(x)\dfrac{\sum_{j=1}^{n}\varphi_j(x)\mathfrak{E}_j(\varphi_j u)(x)}{\sum_{j=1}^{n}\varphi_j(x)^2}.$$
The second term is well-defined, since for every $x\in \text{supp }\Phi$ we have $\varphi_j(x)=1$ for at least one $j$. Since $ \mathfrak{E}u(x)=u(x)$ for $x\in K'$ (according to the definition of the functions $\Psi$ and $\Phi$) and 
\begin{equation}\label{supp-ext}
\mathfrak{E}u=0\quad \text{on} \quad \{x\in \mathbb{R}^3\setminus K': \: \mathrm{dist}(x,\partial K')>\varepsilon/2 \},  
\end{equation}
 we see that $\mathfrak{E}$ indeed defines an extension from the layer $K'$ to $\mathbb{R}^3$, which is continuous because every term of $\mathfrak{E}$ is continuous.\\
It remains to verify the corresponding norm estimate in the spaces $V_{\beta,\delta}^{l,p}$ with $\delta=(\delta_1,...,\delta_n)$, i.e.,
 \begin{align}
     \label{okt15c}
     \|\mathfrak{E}u|V_{\beta,\delta}^{l,p}(\mathbb{R}^3,S) \|\lesssim  \|u|V_{\beta,\delta}^{l,p}(K',S')\|.
 \end{align}
The functions $\Psi,\Phi,\varphi_j$ are smooth with compact supports, and are therefore multipliers in the spaces $ V_{\beta,\delta}^{l,p}$, cf. \eqref{okt15d}.
An easy calculation gives 
 \begin{align*}
\|\mathfrak{E}u|V_{\beta,\delta}^{l,p}(\mathbb{R}^3,S) \|& \ =\ \left\|
 \Psi u+\Phi\dfrac{\sum_{j=1}^{n}\varphi_j\mathfrak{E}_j(\varphi_j u)}{\sum_{j=1}^{n}\varphi_j^2}
 \Big|V_{\beta,\delta}^{l,p}(\mathbb{R}^3,S) \right\|\\
 & \ \lesssim \ \|  \Psi u| V_{\beta,\delta}^{l,p}(\mathbb{R}^3,S) \|     +\left\| \sum_{j=1}^{n}\varphi_j\mathfrak{E}_j(\varphi_j u)   | V_{\beta,\delta}^{l,p}(\mathbb{R}^3,S)  \right\|\\
 & \ \lesssim \  \|u|V_{\beta,\delta}^{l,p}(K',S) \|   
 +\displaystyle\sum_{j=1}^{n} \| \varphi_j\mathfrak{E}_j(\varphi_j u)  | V_{\beta,\delta}^{l,p}(\mathbb{R}^3,S)\| \  =: \ I+II. 
 \end{align*}
Since for the first term $I$ we have
 $ \|u|V_{\beta,\delta}^{l,p}(K',S) \|\sim \|u|V_{\beta,\delta}^{l,p}(K',S') \|$ 
it suffices to consider the second term $II$. In the estimation of $II$ we will use the following norm estimates of the operators for single wedges,
 \begin{align*}
 \|\mathfrak{E}_k u|V_{\delta_k}^{l,p}(\mathbb{R}^3,M_k)\|\lesssim  \|u|V_{\delta_k}^{l,p}(W_k,M_k)\| \sim \|u|V_{\delta_k}^{l,p}(W_k,M_k\cap \overline{ W_k})\|\qquad \forall \   k\in \{1,...,n   \}.
  \end{align*}
  Moreover,  by \eqref{supp-ext} we have $\mathrm{supp}(\mathfrak{E} u)\subset \bigcup_{j=1}^n U_j$, i.e., the support  is contained in a small neighbourhood of $K'$. Since on $U_j$ (neighborhood of  the wedge $W_j$) we have for the distance  functions $\rho_0, r_1,...,r_{j-1},r_{j+1},...,r_n\sim 1$, it follows that  on $U_j$ the norms of $V_{\delta_j}^{l,p}$ and $V_{\beta,\delta}^{l,p}$ are equivalent w.r.t. $S$. A similar statement holds for $\rho_0',r'_{1},...,r'_{n}$ replacing $S$ by $S'$. This ultimately yields
  \begin{align*}
 \displaystyle\sum_{j=1}^{n} &\| \varphi_j\mathfrak{E}_j(\varphi_j u)  | V_{\beta,\delta}^{l,p}(\mathbb{R}^3,S) \| \\ 
 &\sim \displaystyle\sum_{j=1}^{n} \| \varphi_j\mathfrak{E}_j(\varphi_j u)  | V_{\delta_j}^{l,p}(\mathbb{R}^3,M_j) \|     
 \lesssim \displaystyle\sum_{j=1}^{n} \| \mathfrak{E}_j(\varphi_j u)  | V_{\delta_j}^{l,p}(\mathbb{R}^3,M_j) \| \\
  &\lesssim \displaystyle\sum_{j=1}^{n} \| \varphi_j u  | V_{\delta_j}^{l,p}(W_j, \overline{W_j}\cap M_j) \| 
  \lesssim \displaystyle\sum_{j=1}^{n} \| \varphi_j u  | V_{\beta,\delta}^{l,p}(K', S') \|
  \lesssim \|u|V_{\beta,\delta}^{l,p}(K', S')  \|. 
  \end{align*}
 Therefore,   \eqref{okt15c} holds, which finally completes the proof.
\end{proof}

 \subsection{Extension operator for the cone}

 Having established an extension operator for our weighted Sobolev spaces on the fixed layer $K'\subset K$ of the cone in Lemma \ref{layerextopmod}   we are now going to construct an extension operator on the whole cone $K$ as follows. We first  decompose the cone $K$   into several layers $K_j$ and use extension operators on each layer (which are uniformly bounded w.r.t. $j$). Afterwards we   glue these extension operators together in a suitable way to obtain an extension for the whole cone. 
 For this let 
 	\begin{align}
 	\label{layerofthecone}
 	    K_j=\{ x\in K: 2^{-j-1}<|x|<2^{-j+1}  \}
 	\end{align}
 	be a (dyadic) layer of the cone. The following lemma shows that on these layers the distance of a point $x\in K_j$ to the restricted  set $S_j:=\overline{K_j}\cap  S$ is equivalent to its distance to the whole set $S$.

  	\medskip 
 
 \begin{lemma} \label{equiv-dist}   
For  the layer  $E_j:=\{
 x\in \real^3: 2^{-j-1}<|x|<2^{-j+1}\}$ and the set  $S_j$ defined above consider the respective distance function on $E_j$ w.r.t. the edge $M_k$, 
 	$$r_{k,j}(x):= \mathrm{dist}(x,S_j\cap M_k).$$
 Then we have 
 $$r_{k,j}(x)\sim r_k(x) \qquad \text{for all } \quad x\in  E_j, $$
 i.e., the distance of a point $x\in E_j$ to the set $S_j\cap M_k$ is equivalent to the distance of $x$ to the whole edge $M_k$. 
 \end{lemma}

 \begin{proof}

We decompose $x$ as $x=x_{\parallel}+x_{\perp}$ with $x_{\parallel}  $ being the component parallel to $M_k$ and consider the difference
$$ \Delta=2^{-j-1}-|x_{\parallel}|.$$\\

\begin{minipage}{0.55\textwidth}
\textit{Case 1:} Let $x\in E_j$ and $\Delta>0$, i.e.  $$|x_{\parallel}|<2^{-j-1}.$$
 It is clear that 
\begin{align}
\label{csillagaug13}
r_k(x)\leq r_{k,j}(x).
\end{align}
By geometric considerations we see that for $x\in E_j$ it holds 
$$r_k^2(x)+|x_{\parallel}|^2=|x|^2\geq (2^{-j-1})^2=(\Delta+|x_{\parallel}|)^2, $$
hence
\begin{align}
\label{egy3}
r_k^2(x)\geq \Delta^2+2\Delta |x_{\parallel}| \geq \Delta^2.
\end{align}
	\end{minipage}\hfill \begin{minipage}{0.4\textwidth}
	\centering              
	\def\svgwidth{210pt}    
	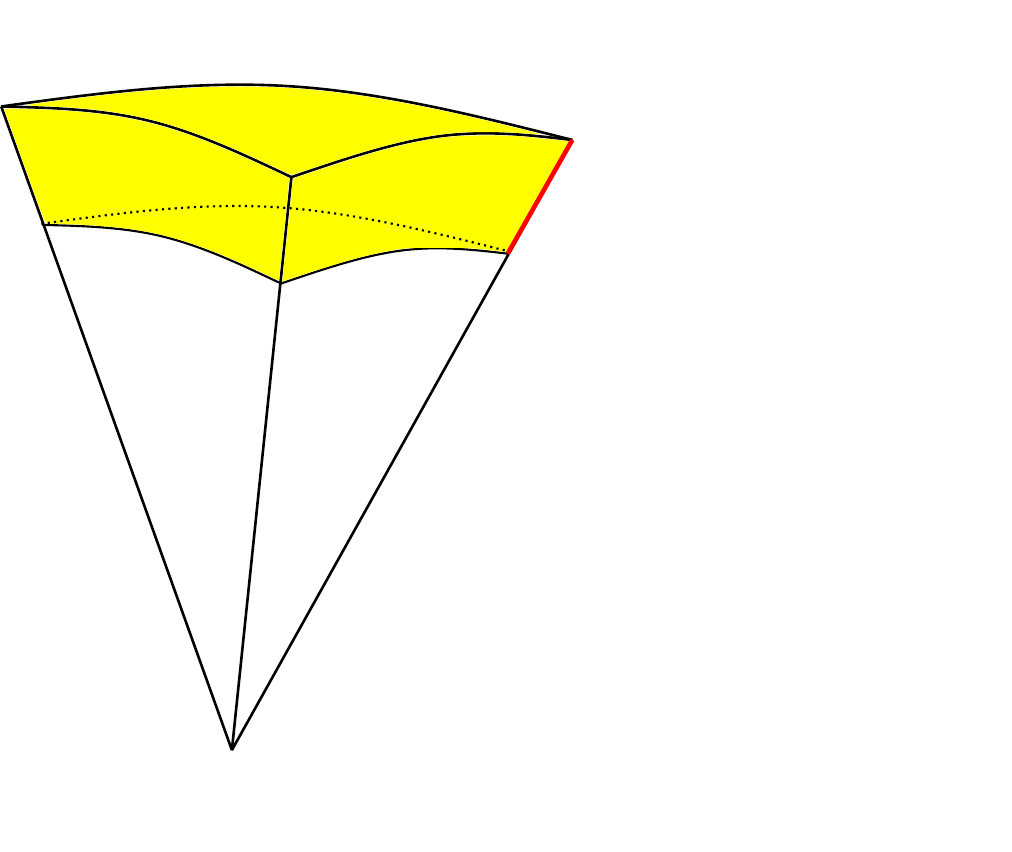  
	\end{minipage}\\[0.3cm]

Since $r_{k,j}^2(x)=r_k^2(x)+\Delta^2$, using \eqref{egy3} we obtain 
$
\ r_{k,j}^2(x)\leq 2 r_k^2(x),\ 
$
i.e.,
\begin{align}
\label{ketto2}
r_{k,j}(x)\lesssim r_k(x).
\end{align}
Therefore \eqref{csillagaug13} and \eqref{ketto2} imply that
\begin{align*}
r_{k,j}(x) \sim r_{k}(x) \qquad \text{for} \quad x\in E_j, \  \Delta>0.
\end{align*}

	\begin{minipage}{0.5\textwidth}
	\textit{Case 2:} Let $x\in E_j$ and $\Delta\leq0$, then $$2^{-j-1}\leq |x_{\parallel}|\leq 2^{-j+1}.$$ 
	Obviously in this case we have
$  r_k(x)=r_{k,j}(x).$  
	\end{minipage}\hfill \begin{minipage}{0.45\textwidth}
	\def\svgwidth{220pt}    
	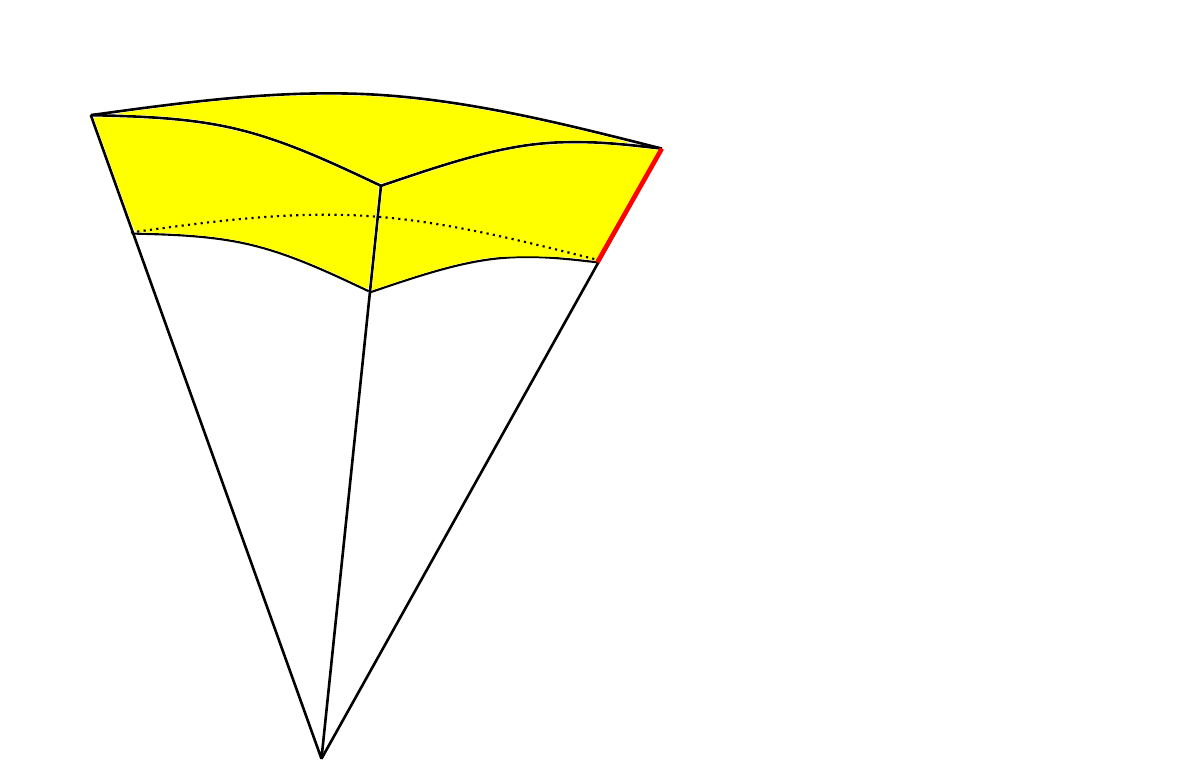  
	\end{minipage}

Together  Cases 1 and 2 yield 
$\ r_{k,j}\sim r_k \ \text{ on } \ E_j. $ 
 \end{proof}

 \begin{remark}
 	\label{remark5}
 	\bit 
 \item[(i)] 	Below we use the notation $V_{\beta,\delta}^{l,p}(K_j,S_j)$. This is to be understood in the same way  as already explained in Remark \ref{notation-expl}. In particular, in view of the definition of the spaces  
 $V_{\beta,\delta}^{l,p}(K,S)$, we replace the weight functions $r_k$ by $r_{k,j}$ (which according to Lemma \ref{equiv-dist} are equivalent on the layers $K_j$), and the integral domain $K$ by $K_j$ in the  norm, respectively.
 	
 \item[(ii)] 	\hfill 
 
 \begin{minipage}{0.5\textwidth} \vspace{-4.5cm}
 Let $\theta$ be the angle of the edge $M_k$ and of the line $\overrightarrow{0x}$. 
 	Then \begin{align*}
 	    \sin \theta=\frac{r_k(x)}{\rho_0 (x)}=\frac{r_k(2^jx)}{\rho_0 (2^jx)}
 	\end{align*} and we see that the distance functions 
 	\begin{align*}
 	r_k(2^jx)=2^j|x|\sin\theta =2^j r_k(x)
 	\end{align*}
 	and 
 	\begin{align*}
 	\rho_0(2^jx)=|2^jx|=2^j|x|=2^j \rho_0(x)
 	\end{align*}
 are homogeneous w.r.t. isotropic dilations. 
 
 \end{minipage}\hfill 
 	\begin{minipage}[b]{0.3\textwidth}
 		\def\svgwidth{150pt}
 		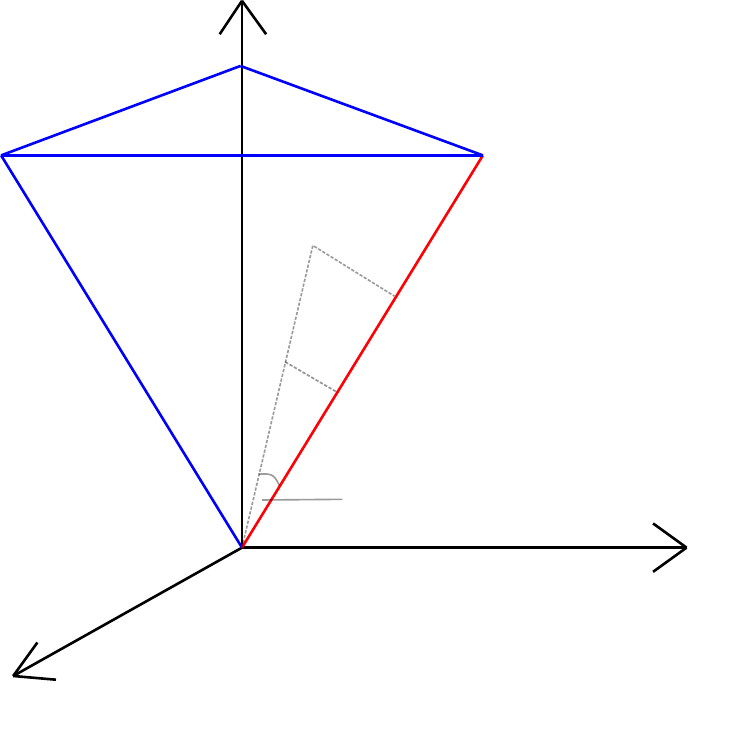
 	\end{minipage} 
 
 \eit 
 \end{remark}

 After these preparations we can prove Theorem \ref{extop}. 
 
 \begin{proof}[Proof of Theorem \ref{extop}]
 	We decompose $K$ into the dyadic layers $K_j$ from \eqref{layerofthecone},
 	$$K=\displaystyle\bigcup_{j\geq j_0} K_j, $$
where $j_0\in \mathbb{Z}$ is fixed and depends on the cone.\\
\textit{Step 1:} Let $\mathfrak{E}_0$ denote the extension operator from Lemma \ref{layerextopmod} with regard to the layer $K_0$. Then we define an extension operator on $K_j$ via
$$\mathfrak{E}_j=T_{-j}\circ \mathfrak{E}_0\circ T_j, $$
where the dilation operator $T_j$ is defined as $T_ju(x)=u(2^{j}x)$. The inverse operator is given by $T_{-j}=T_j^{-1}.$ 
In order to prove uniform boundedness of the family of extension operators $(\mathfrak{E}_j)_j$ we first need to calculate the operator norms of the family of dilation operators $(T_j)_j$.  In particular, for $u\in  V_{\beta,\delta}^{l,p}(K_0,S_0)$ using the homogeneity of $r_k$ and $\rho_0$, we compute 
%
\begin{align*}
    & \|T_ju| V_{\beta,\delta}^{l,p}(K_j,S_j)  \|^p\\
    & \sim \displaystyle\sum_{|\alpha|\leq l}\int_{K_j}\rho_0(x)^{(\beta-l+|\alpha|)p}\prod_{k=1}^{n} \left(  \frac{r_k(x)}{\rho_0(x)}  \right)^{(\delta_k-l+|\alpha|)p}2^{j|\alpha|p} |(\partial^{\alpha}u)(2^jx)   |^p\: dx\\
    & =\sum_{|\alpha|\leq l}\int_{K_j} 2^{-j(\beta-l+|\alpha|)p} \rho_0(2^jx)^{(\beta-l+|\alpha|)p}\prod_{k=1}^{n} \left(  \frac{r_k(2^jx)}{\rho_0(2^jx)}  \right)^{(\delta_k-l+|\alpha|)p}2^{j|\alpha|p} |(\partial^{\alpha}u)(2^jx)   |^p\: dx\\
    & = 2^{-j(\beta-l)p-j3}\|u|V_{\beta,\delta}^{l,p}(K_0,S_0)\|^p,
\end{align*}
with the integral substitution $y:=2^jx$ and $dy=2^{j3}dx$.   Hence,
\begin{align}
\label{nov4du}
    \|T_j \|\sim 2^{-j\left(\beta-l+\frac 3p\right)}\qquad \forall \ j\in \mathbb{Z}.  
\end{align} 

A similar estimate holds for $T_{-j}.$ 
We conclude
\begin{align*}
\|\mathfrak{E}_j:V_{\beta,\delta}^{l,p} (K_j,S_j)\rightarrow V_{\beta,\delta}^{l,p} (\mathbb{R}^3,S)\| &\lesssim \|T_{-j}||\cdot \|\mathfrak{E}_0:V_{\beta,\delta}^{l,p} (K_0,S_0)\rightarrow V_{\beta,\delta}^{l,p} (\mathbb{R}^3,S)\|\cdot \|T_j\|\\
& =
\|\mathfrak{E}_0:V_{\beta,\delta}^{l,p} (K_0,S_0)\rightarrow V_{\beta,\delta}^{l,p} (\mathbb{R}^3,S)\|,
\end{align*}
independent of $j\in\mathbb{Z}$. Since $\mathfrak{E}_0$ is a bounded operator according to Lemma \ref{layerextopmod}, $(\mathfrak{E}_j)_{j\in \mathbb{Z}}$ is a family of uniformly bounded operators and 
\begin{align*}
    \|\mathfrak{E}_j \varphi |V_{\beta,\delta}^{l,p} (\mathbb{R}^3,S) \| \lesssim \|\varphi|V_{\beta,\delta}^{l,p} (K_j,S_j)\| \qquad \forall \ j\in \mathbb{Z}.
\end{align*}
\textit{Step 2:}  We define the extension operator $\mathfrak{E}$ via a suitable combination of the operators $\mathfrak{E}_j$.\\
\textit{Substep 2.1:} Let $\varphi_0:\mathbb{R}^3\longrightarrow [0,1]$ be a radially symmetric smooth function, such that 
\begin{align*}
    \text{supp}\: \varphi_0 \subset  \left\{\frac{1}{2} < |x|<2  \right\} =: E_0,
\end{align*}
and
\begin{align*}
         \varphi_0  \equiv 1 \quad \text{on}\quad \left\{\frac{1}{2}+\varepsilon<|x|<2-\varepsilon   \right\},
\end{align*}
for some $\varepsilon>0$ sufficiently small. 
Moreover, we put 
$ 
    \varphi_j(x):=\varphi_0(2^jx),
$ 
for which we have
\begin{align*}
     \text{supp}\: \varphi_j \subset  \{2^{-j-1} < |x|<2^{-j+1} \} =: E_j,
\end{align*}
and
\begin{align*}
 \varphi_j  \equiv 1 \quad \text{on}\quad \left\{2^{-j}\left(\frac{1}{2}+\varepsilon\right)<|x|<2^{-j}(2-\varepsilon )  \right\}.    
\end{align*}
From the construction it follows that
\begin{align}
\label{nov3cond1}
    \displaystyle\sum_{j\in \mathbb{Z}} \varphi_j(x) \sim 1 \quad \text{for all}\quad x\in \mathbb{R}^3.
\end{align}
Furthermore, since $\varphi_0$ is smooth and has compact support, we see that 
\begin{align}
    \label{nov3cond2}
 | \partial^{\alpha} \varphi_0 | \leq c_{\alpha} \qquad \text{and} \qquad  |\partial^{\alpha} \varphi_j(x)|= |\partial^{\alpha}(\varphi_0(2^jx))   |\leq c_{\alpha}2^{j|\alpha|},
\end{align}
for some constant $c_{\alpha}$. 
By \eqref{nov3cond1} and \eqref{nov3cond2}   the family $(\varphi_j)_{j\in \mathbb{Z}}$ forms a suitable localization for the cone $K$, where $\text{supp }\varphi_j|_K\subset K_j$, and we can apply Lemma \ref{lemma2}.\\

\textit{Substep 2.2:} Now we put
$$\Phi(x)=\displaystyle\sum_{j\in \mathbb{Z}}\varphi_j(x)^2,\quad x\neq 0, $$
and define the extension operator via 
$$\mathfrak{E}u(x):=\frac{1}{\Phi(x)} \displaystyle\sum_{j\in \mathbb{Z}} \varphi_j(x)\mathfrak{E}_j(\varphi_j|_K u)(x). $$
The mapping $\mathfrak{E}$ is well-defined because from the construction we deduce  $ \Phi(x)\sim 1$ for all $x\neq 0$. 
Moreover, it is immediately verified that $\mathfrak{E}u$ defines an extension of $u$ to $ \mathbb{R}^3$.

\textit{Step 3:} It remains to show that
\begin{align}
    \label{nov4} 
    \|\mathfrak{E}u|V_{\beta,\delta}^{l,p}(\mathbb{R}^3,S)  \| \lesssim \|u|V_{\beta,\delta}^{l,p}(K,S) \|.
\end{align} 
For the calculations to come we use the fact that the localization  of Lemma \ref{lemma2} also works for the spaces $V_{\beta,\delta}^{l,p}(\mathbb{R}^3,S)$  
and that the functions $\varphi_j,\: j\in \mathbb{Z}$, are multipliers in the spaces $V_{\beta,\delta}^{l,p}$. 
With this we estimate 
\begin{align*}
   \|\mathfrak{E}& u|V_{\beta,\delta}^{l,p}(\mathbb{R}^3,S)\|^p  \\
   & \sim \displaystyle\sum_{j\geq j_0}  \|\varphi_j\mathfrak{E} u|V_{\beta,\delta}^{l,p}(\mathbb{R}^3,S)\|^p 
    \sim \displaystyle\sum_{j\geq j_0}  2^{-j(\beta-l)p-j3} \|(\varphi_j \mathfrak{E}u)(2^{-j}\cdot)|V_{\beta,\delta}^{l,p}(\mathbb{R}^3,S)\|^p\\
   & = \displaystyle\sum_{j\geq j_0}  2^{-j(\beta-l)p-j3} \| \varphi_0(\cdot) \mathfrak{E}u (2^{-j}\cdot)|V_{\beta,\delta}^{l,p}(\mathbb{R}^3,S)\|^p\\
   &=\displaystyle\sum_{j\geq j_0} 2^{-j(\beta-l)p-j3}
   \left\|
   \underbrace{\varphi_0(x)}_{\text{supp}\: \cdot \subset E_0} \underbrace{\frac{1}{\Phi(2^{-j}x)}}_{\sim 1} \displaystyle\sum_{m\in\mathbb{Z}} \underbrace{\varphi_m(2^{-j}x)}_{\substack{\neq 0\:\text{for}\\ |m-j|\leq 1}}\mathfrak{E}_m(\varphi_m|_K u)(2^{-j}x)
   \bigg|V_{\beta,\delta}^{l,p}(\mathbb{R}^3,S) \right\|^p\\
      &\lesssim\displaystyle\sum_{j\geq j_0} 2^{-j(\beta-l)p-j3}\|\mathfrak{E}_j(\varphi_j|_K u)(2^{-j}\cdot) |V_{\beta,\delta}^{l,p}(\mathbb{R}^3,S)\|^p\\
        &\sim\displaystyle\sum_{j\geq j_0} \|\mathfrak{E}_j(\varphi_j|_K u) |V_{\beta,\delta}^{l,p}(\mathbb{R}^3,S)\|^p\\
&\lesssim \sum_{j\geq j_0} 
\|\varphi_j|_K u |V_{\beta,\delta}^{l,p}(K_j,S_j)\|^p
\sim \sum_{j\geq j_0} 
\|\varphi_j u |V_{\beta,\delta}^{l,p}(K,S)\|^p
 \sim \| u |V_{\beta,\delta}^{l,p}(K,S)\|^p,
\end{align*}
which completes the proof.
 \end{proof}


\end{document}

%% file: infinitecone.pdf_tex
\begingroup%
  \makeatletter%
  \providecommand\color[2][]{%
    \errmessage{(Inkscape) Color is used for the text in Inkscape, but the package 'color.sty' is not loaded}%
    \renewcommand\color[2][]{}%
  }%
  \providecommand\transparent[1]{%
    \errmessage{(Inkscape) Transparency is used (non-zero) for the text in Inkscape, but the package 'transparent.sty' is not loaded}%
    \renewcommand\transparent[1]{}%
  }%
  \providecommand\rotatebox[2]{#2}%
  \ifx\svgwidth\undefined%
    \setlength{\unitlength}{363.30899767bp}%
    \ifx\svgscale\undefined%
      \relax%
    \else%
      \setlength{\unitlength}{\unitlength * \real{\svgscale}}%
    \fi%
  \else%
    \setlength{\unitlength}{\svgwidth}%
  \fi%
  \global\let\svgwidth\undefined%
  \global\let\svgscale\undefined%
  \makeatother%
  \begin{picture}(1,0.65814594)%
    \put(0,0){\includegraphics[width=\unitlength,page=1]{infinitecone.pdf}}%
    \put(0.89033096,0.44420763){\color[rgb]{0,0,0}\makebox(0,0)[lb]{\smash{$M_1$}}}%
    \put(0.29903972,0.5400529){\color[rgb]{0,0,0}\makebox(0,0)[lb]{\smash{$M_2$}}}%
    \put(0.04394398,0.34246439){\color[rgb]{0,0,0}\makebox(0,0)[lb]{\smash{$M_3$}}}%
    \put(0.60574429,0.48549478){\color[rgb]{0,0,0}\makebox(0,0)[lb]{\smash{$M_4$}}}%
    \put(0.46861194,0.35720978){\color[rgb]{0,0,0}\makebox(0,0)[lb]{\smash{$\Omega$}}}%
    \put(0,0){\includegraphics[width=\unitlength,page=2]{infinitecone.pdf}}%
  \end{picture}%
\endgroup%

%% file: boundedcone.pdf_tex
\begingroup%
  \makeatletter%
  \providecommand\color[2][]{%
    \errmessage{(Inkscape) Color is used for the text in Inkscape, but the package 'color.sty' is not loaded}%
    \renewcommand\color[2][]{}%
  }%
  \providecommand\transparent[1]{%
    \errmessage{(Inkscape) Transparency is used (non-zero) for the text in Inkscape, but the package 'transparent.sty' is not loaded}%
    \renewcommand\transparent[1]{}%
  }%
  \providecommand\rotatebox[2]{#2}%
  \ifx\svgwidth\undefined%
    \setlength{\unitlength}{253.33394186bp}%
    \ifx\svgscale\undefined%
      \relax%
    \else%
      \setlength{\unitlength}{\unitlength * \real{\svgscale}}%
    \fi%
  \else%
    \setlength{\unitlength}{\svgwidth}%
  \fi%
  \global\let\svgwidth\undefined%
  \global\let\svgscale\undefined%
  \makeatother%
  \begin{picture}(1,0.96032239)%
    \put(0,0){\includegraphics[width=\unitlength,page=1]{boundedcone.pdf}}%
    \put(0.67057333,0.20664158){\color[rgb]{0,0,0}\makebox(0,0)[lb]{\smash{$M_1$}}}%
    \put(0.37053939,0.78026381){\color[rgb]{0,0,0}\makebox(0,0)[lb]{\smash{$M_2$}}}%
    \put(0.07968217,0.25366898){\color[rgb]{0,0,0}\makebox(0,0)[lb]{\smash{$M_3$}}}%
    \put(0.57763644,0.51265146){\color[rgb]{0,0,0}\makebox(0,0)[lb]{\smash{$M_4$}}}%
    \put(0,0){\includegraphics[width=\unitlength,page=2]{boundedcone.pdf}}%
    \put(0.24728402,0.36633042){\color[rgb]{0,0,0}\makebox(0,0)[lb]{\smash{$\rho_0(x)$}}}%
    \put(0.06542357,0.44034334){\color[rgb]{0,0,0}\makebox(0,0)[lb]{\smash{$r_3(x)$}}}%
    \put(0.61311959,0.69833146){\color[rgb]{0,0,0}\makebox(0,0)[lb]{\smash{$\Omega$}}}%
    \put(0,0){\includegraphics[width=\unitlength,page=3]{boundedcone.pdf}}%
    \put(0.23036672,0.51012713){\color[rgb]{0,0,0}\makebox(0,0)[lb]{\smash{$x$}}}%
    \put(0,0){\includegraphics[width=\unitlength,page=4]{boundedcone.pdf}}%
  \end{picture}%
\endgroup%

%% file: speclipdom.pdf_tex
\begingroup%
  \makeatletter%
  \providecommand\color[2][]{%
    \errmessage{(Inkscape) Color is used for the text in Inkscape, but the package 'color.sty' is not loaded}%
    \renewcommand\color[2][]{}%
  }%
  \providecommand\transparent[1]{%
    \errmessage{(Inkscape) Transparency is used (non-zero) for the text in Inkscape, but the package 'transparent.sty' is not loaded}%
    \renewcommand\transparent[1]{}%
  }%
  \providecommand\rotatebox[2]{#2}%
  \ifx\svgwidth\undefined%
    \setlength{\unitlength}{263.57658904bp}%
    \ifx\svgscale\undefined%
      \relax%
    \else%
      \setlength{\unitlength}{\unitlength * \real{\svgscale}}%
    \fi%
  \else%
    \setlength{\unitlength}{\svgwidth}%
  \fi%
  \global\let\svgwidth\undefined%
  \global\let\svgscale\undefined%
  \makeatother%
  \begin{picture}(1,0.51412955)%
    \put(0,0){\includegraphics[width=\unitlength,page=1]{speclipdom.pdf}}%
    \put(0.08831072,0.45400875){\color[rgb]{0,0,0}\makebox(0,0)[lb]{\smash{$x_3$}}}%
    \put(0.78873145,0.32594205){\color[rgb]{0,0,0}\makebox(0,0)[lb]{\smash{$x_2$}}}%
    \put(0.87194632,0.02410914){\color[rgb]{0,0,0}\makebox(0,0)[lb]{\smash{$x_1$}}}%
    \put(0.41564223,0.00735636){\color[rgb]{1,0,0}\makebox(0,0)[lb]{\smash{$S$}}}%
    \put(0,0){\includegraphics[width=\unitlength,page=2]{speclipdom.pdf}}%
  \end{picture}%
\endgroup%

%% file: gammaxiprime.pdf_tex
\begingroup%
  \makeatletter%
  \providecommand\color[2][]{%
    \errmessage{(Inkscape) Color is used for the text in Inkscape, but the package 'color.sty' is not loaded}%
    \renewcommand\color[2][]{}%
  }%
  \providecommand\transparent[1]{%
    \errmessage{(Inkscape) Transparency is used (non-zero) for the text in Inkscape, but the package 'transparent.sty' is not loaded}%
    \renewcommand\transparent[1]{}%
  }%
  \providecommand\rotatebox[2]{#2}%
  \ifx\svgwidth\undefined%
    \setlength{\unitlength}{272.35209798bp}%
    \ifx\svgscale\undefined%
      \relax%
    \else%
      \setlength{\unitlength}{\unitlength * \real{\svgscale}}%
    \fi%
  \else%
    \setlength{\unitlength}{\svgwidth}%
  \fi%
  \global\let\svgwidth\undefined%
  \global\let\svgscale\undefined%
  \makeatother%
  \begin{picture}(1,0.8029077)%
    \put(0,0){\includegraphics[width=\unitlength,page=1]{gammaxiprime.pdf}}%
    \put(0.89683816,0.02353733){\color[rgb]{0,0,0}\makebox(0,0)[lb]{\smash{$x'$}}}%
    \put(-0.00305678,0.77182719){\color[rgb]{0,0,0}\makebox(0,0)[lb]{\smash{$x_3$}}}%
    \put(0.27372698,0.66890245){\color[rgb]{0,0,0}\makebox(0,0)[lb]{\smash{$D$}}}%
    \put(0.18316411,0.34880077){\color[rgb]{0,0,0}\makebox(0,0)[lb]{\smash{$\Gamma_{\xi^0}$}}}%
    \put(0.58250087,0.37681917){\color[rgb]{0,0,0}\makebox(0,0)[lb]{\smash{}}}%
    \put(0.39708533,0.51493172){\color[rgb]{0,0,0}\makebox(0,0)[lb]{\smash{$\xi^0=(x^0,\omega(x^0))$}}}%
    \put(0.54527286,0.16652991){\color[rgb]{0,0,0}\makebox(0,0)[lb]{\smash{$q$}}}%
    \put(0.44658424,0.40449692){\color[rgb]{0,0,0}\makebox(0,0)[lb]{\smash{$c$}}}%
    \put(0,0){\includegraphics[width=\unitlength,page=2]{gammaxiprime.pdf}}%
    \put(0.42693882,0.13119067){\color[rgb]{0,0,0}\makebox(0,0)[lb]{\smash{$b$}}}%
    \put(0.31251589,0.167856){\color[rgb]{0,0,0}\makebox(0,0)[lb]{\smash{$p$}}}%
    \put(0.32765632,0.31257107){\color[rgb]{0,0,0}\makebox(0,0)[lb]{\smash{$a$}}}%
    \put(0.40830298,0.21283063){\color[rgb]{0,0,0}\makebox(0,0)[lb]{\smash{$h$}}}%
    \put(0,0){\includegraphics[width=\unitlength,page=3]{gammaxiprime.pdf}}%
    \put(0.35370209,0.0499641){\color[rgb]{0,0,0}\makebox(0,0)[lb]{\smash{$x^0$}}}%
    \put(0,0){\includegraphics[width=\unitlength,page=4]{gammaxiprime.pdf}}%
    \put(0.04701466,0.16610206){\color[rgb]{0,0,0}\makebox(0,0)[lb]{\smash{$y$}}}%
  \end{picture}%
\endgroup%

%% file: layerD.pdf_tex
\begingroup%
  \makeatletter%
  \providecommand\color[2][]{%
    \errmessage{(Inkscape) Color is used for the text in Inkscape, but the package 'color.sty' is not loaded}%
    \renewcommand\color[2][]{}%
  }%
  \providecommand\transparent[1]{%
    \errmessage{(Inkscape) Transparency is used (non-zero) for the text in Inkscape, but the package 'transparent.sty' is not loaded}%
    \renewcommand\transparent[1]{}%
  }%
  \providecommand\rotatebox[2]{#2}%
  \ifx\svgwidth\undefined%
    \setlength{\unitlength}{311.00804817bp}%
    \ifx\svgscale\undefined%
      \relax%
    \else%
      \setlength{\unitlength}{\unitlength * \real{\svgscale}}%
    \fi%
  \else%
    \setlength{\unitlength}{\svgwidth}%
  \fi%
  \global\let\svgwidth\undefined%
  \global\let\svgscale\undefined%
  \makeatother%
  \begin{picture}(1,1.57383604)%
    \put(0,0){\includegraphics[width=\unitlength,page=1]{layerD.pdf}}%
    \put(0.620936,0.75905497){\color[rgb]{0.98039216,0,0}\makebox(0,0)[lb]{\smash{$S'$}}}%
    \put(0,0){\includegraphics[width=\unitlength,page=2]{layerD.pdf}}%
    \put(0.42307456,0.92134493){\color[rgb]{0,0,0}\makebox(0,0)[lb]{\smash{$K'$}}}%
    \put(-0.00730048,0.65711927){\color[rgb]{0,0,0}\makebox(0,0)[lb]{\smash{$C_1$}}}%
    \put(0.0036615,1.0712387){\color[rgb]{0,0,0}\makebox(0,0)[lb]{\smash{$C_2$}}}%
  \end{picture}%
\endgroup%

%% file: omega.pdf_tex
\begingroup%
  \makeatletter%
  \providecommand\color[2][]{%
    \errmessage{(Inkscape) Color is used for the text in Inkscape, but the package 'color.sty' is not loaded}%
    \renewcommand\color[2][]{}%
  }%
  \providecommand\transparent[1]{%
    \errmessage{(Inkscape) Transparency is used (non-zero) for the text in Inkscape, but the package 'transparent.sty' is not loaded}%
    \renewcommand\transparent[1]{}%
  }%
  \providecommand\rotatebox[2]{#2}%
  \ifx\svgwidth\undefined%
    \setlength{\unitlength}{206.5777169bp}%
    \ifx\svgscale\undefined%
      \relax%
    \else%
      \setlength{\unitlength}{\unitlength * \real{\svgscale}}%
    \fi%
  \else%
    \setlength{\unitlength}{\svgwidth}%
  \fi%
  \global\let\svgwidth\undefined%
  \global\let\svgscale\undefined%
  \makeatother%
  \begin{picture}(1,0.96737206)%
    \put(0,0){\includegraphics[width=\unitlength,page=1]{omega.pdf}}%
    \put(-0.0078852,0.66994592){\color[rgb]{0,0,0}\makebox(0,0)[lb]{\smash{1}}}%
    \put(-0.00269859,0.00606563){\color[rgb]{0,0,0}\makebox(0,0)[lb]{\smash{}}}%
    \put(-0.00269859,0.00087928){\color[rgb]{0,0,0}\makebox(0,0)[lb]{\smash{0}}}%
    \put(0,0){\includegraphics[width=\unitlength,page=2]{omega.pdf}}%
    \put(0.29812212,0.055338){\color[rgb]{0,0,0}\makebox(0,0)[lb]{\smash{$\varphi_0$}}}%
    \put(0,0){\includegraphics[width=\unitlength,page=3]{omega.pdf}}%
    \put(0.38629365,0.37949826){\color[rgb]{0,0,0}\makebox(0,0)[lb]{\smash{$\Omega$}}}%
  \end{picture}%
\endgroup%

%% file: omegaplus.pdf_tex
\begingroup%
  \makeatletter%
  \providecommand\color[2][]{%
    \errmessage{(Inkscape) Color is used for the text in Inkscape, but the package 'color.sty' is not loaded}%
    \renewcommand\color[2][]{}%
  }%
  \providecommand\transparent[1]{%
    \errmessage{(Inkscape) Transparency is used (non-zero) for the text in Inkscape, but the package 'transparent.sty' is not loaded}%
    \renewcommand\transparent[1]{}%
  }%
  \providecommand\rotatebox[2]{#2}%
  \ifx\svgwidth\undefined%
    \setlength{\unitlength}{193.14228443bp}%
    \ifx\svgscale\undefined%
      \relax%
    \else%
      \setlength{\unitlength}{\unitlength * \real{\svgscale}}%
    \fi%
  \else%
    \setlength{\unitlength}{\svgwidth}%
  \fi%
  \global\let\svgwidth\undefined%
  \global\let\svgscale\undefined%
  \makeatother%
  \begin{picture}(1,0.71319452)%
    \put(0,0){\includegraphics[width=\unitlength,page=1]{omegaplus.pdf}}%
    \put(0.11915131,0.27437788){\color[rgb]{0,0,0}\makebox(0,0)[lb]{\smash{$\Omega_+$}}}%
    \put(0.2633824,0.05525738){\color[rgb]{0,0,0}\makebox(0,0)[lb]{\smash{$\varphi_0$}}}%
  \end{picture}%
\endgroup%

%% file: omeganul.pdf_tex
\begingroup%
  \makeatletter%
  \providecommand\color[2][]{%
    \errmessage{(Inkscape) Color is used for the text in Inkscape, but the package 'color.sty' is not loaded}%
    \renewcommand\color[2][]{}%
  }%
  \providecommand\transparent[1]{%
    \errmessage{(Inkscape) Transparency is used (non-zero) for the text in Inkscape, but the package 'transparent.sty' is not loaded}%
    \renewcommand\transparent[1]{}%
  }%
  \providecommand\rotatebox[2]{#2}%
  \ifx\svgwidth\undefined%
    \setlength{\unitlength}{195.16159435bp}%
    \ifx\svgscale\undefined%
      \relax%
    \else%
      \setlength{\unitlength}{\unitlength * \real{\svgscale}}%
    \fi%
  \else%
    \setlength{\unitlength}{\svgwidth}%
  \fi%
  \global\let\svgwidth\undefined%
  \global\let\svgscale\undefined%
  \makeatother%
  \begin{picture}(1,1.08264085)%
    \put(0,0){\includegraphics[width=\unitlength,page=1]{omeganul.pdf}}%
    \put(0.279588,0.42091218){\color[rgb]{0,0,0}\makebox(0,0)[lb]{\smash{$\varphi_0$}}}%
    \put(0.54509804,0.99836132){\color[rgb]{0,0,0}\makebox(0,0)[lb]{\smash{$\Omega_0$}}}%
    \put(0,0){\includegraphics[width=\unitlength,page=2]{omeganul.pdf}}%
  \end{picture}%
\endgroup%

%% file: grafikon.pdf_tex
\begingroup%
  \makeatletter%
  \providecommand\color[2][]{%
    \errmessage{(Inkscape) Color is used for the text in Inkscape, but the package 'color.sty' is not loaded}%
    \renewcommand\color[2][]{}%
  }%
  \providecommand\transparent[1]{%
    \errmessage{(Inkscape) Transparency is used (non-zero) for the text in Inkscape, but the package 'transparent.sty' is not loaded}%
    \renewcommand\transparent[1]{}%
  }%
  \providecommand\rotatebox[2]{#2}%
  \ifx\svgwidth\undefined%
    \setlength{\unitlength}{383.46062944bp}%
    \ifx\svgscale\undefined%
      \relax%
    \else%
      \setlength{\unitlength}{\unitlength * \real{\svgscale}}%
    \fi%
  \else%
    \setlength{\unitlength}{\svgwidth}%
  \fi%
  \global\let\svgwidth\undefined%
  \global\let\svgscale\undefined%
  \makeatother%
  \begin{picture}(1,0.55153957)%
    \put(0,0){\includegraphics[width=\unitlength,page=1]{grafikon.pdf}}%
    \put(0.01296027,0.47934136){\color[rgb]{0,0,0}\makebox(0,0)[lb]{\smash{$\mathbb{R}$}}}%
    \put(0.78729889,0.06410883){\color[rgb]{0,0,0}\makebox(0,0)[lb]{\smash{$\mathbb{R}^{2}$}}}%
    \put(0.45116648,0.0946068){\color[rgb]{0,0,0}\makebox(0,0)[lb]{\smash{$x'$}}}%
    \put(-0.00222761,0.20498151){\color[rgb]{0,0,0}\makebox(0,0)[lb]{\smash{$x_3$}}}%
    \put(0.34151074,0.45095957){\color[rgb]{0,0,0}\makebox(0,0)[lb]{\smash{$G$}}}%
    \put(0.43661598,0.53714313){\color[rgb]{0.88627451,0,0}\makebox(0,0)[lb]{\smash{$G_{x'}^+$}}}%
    \put(0.43769432,0.01103744){\color[rgb]{0,0,0.94509804}\makebox(0,0)[lb]{\smash{$G_{x'}^-$}}}%
    \put(0.53860828,0.34373834){\color[rgb]{0,0,0}\makebox(0,0)[lb]{\smash{$\omega(x')$}}}%
  \end{picture}%
\endgroup%

%% file: setuj.pdf_tex
\begingroup%
  \makeatletter%
  \providecommand\color[2][]{%
    \errmessage{(Inkscape) Color is used for the text in Inkscape, but the package 'color.sty' is not loaded}%
    \renewcommand\color[2][]{}%
  }%
  \providecommand\transparent[1]{%
    \errmessage{(Inkscape) Transparency is used (non-zero) for the text in Inkscape, but the package 'transparent.sty' is not loaded}%
    \renewcommand\transparent[1]{}%
  }%
  \providecommand\rotatebox[2]{#2}%
  \ifx\svgwidth\undefined%
    \setlength{\unitlength}{255.53570965bp}%
    \ifx\svgscale\undefined%
      \relax%
    \else%
      \setlength{\unitlength}{\unitlength * \real{\svgscale}}%
    \fi%
  \else%
    \setlength{\unitlength}{\svgwidth}%
  \fi%
  \global\let\svgwidth\undefined%
  \global\let\svgscale\undefined%
  \makeatother%
  \begin{picture}(1,0.63312371)%
    \put(0,0){\includegraphics[width=\unitlength,page=1]{setuj.pdf}}%
    \put(0.20649909,0.22327044){\color[rgb]{0,0,0}\makebox(0,0)[lb]{\smash{$x$}}}%
    \put(0.10062889,0.1257863){\color[rgb]{0,0,0}\makebox(0,0)[lb]{\smash{$B(x,\varepsilon)$}}}%
    \put(0.43396209,0.05450728){\color[rgb]{0.79607843,0.28235294,0.94509804}\makebox(0,0)[lb]{\smash{$U_j$}}}%
    \put(0.05345912,0.50419284){\color[rgb]{0,0,0}\makebox(0,0)[lb]{\smash{$\partial K'$}}}%
    \put(0,0){\includegraphics[width=\unitlength,page=2]{setuj.pdf}}%
    \put(0.27568135,0.45911943){\color[rgb]{0,0,1}\makebox(0,0)[lb]{\smash{$W_j$}}}%
  \end{picture}%
\endgroup%

%% file: phij.pdf_tex
\begingroup%
  \makeatletter%
  \providecommand\color[2][]{%
    \errmessage{(Inkscape) Color is used for the text in Inkscape, but the package 'color.sty' is not loaded}%
    \renewcommand\color[2][]{}%
  }%
  \providecommand\transparent[1]{%
    \errmessage{(Inkscape) Transparency is used (non-zero) for the text in Inkscape, but the package 'transparent.sty' is not loaded}%
    \renewcommand\transparent[1]{}%
  }%
  \providecommand\rotatebox[2]{#2}%
  \ifx\svgwidth\undefined%
    \setlength{\unitlength}{368.10921101bp}%
    \ifx\svgscale\undefined%
      \relax%
    \else%
      \setlength{\unitlength}{\unitlength * \real{\svgscale}}%
    \fi%
  \else%
    \setlength{\unitlength}{\svgwidth}%
  \fi%
  \global\let\svgwidth\undefined%
  \global\let\svgscale\undefined%
  \makeatother%
  \begin{picture}(1,0.70194534)%
    \put(0,0){\includegraphics[width=\unitlength,page=1]{phij.pdf}}%
    \put(0.00291063,0.61829805){\color[rgb]{0,0,0}\makebox(0,0)[lb]{\smash{$1$}}}%
    \put(0,0){\includegraphics[width=\unitlength,page=2]{phij.pdf}}%
    \put(0.33181146,0.65904687){\color[rgb]{0,0,0}\makebox(0,0)[lb]{\smash{$\varphi_j$}}}%
    \put(0.03783803,0.06673432){\color[rgb]{0,0,0}\makebox(0,0)[lb]{\smash{$\partial K'$}}}%
    \put(0.51809157,0.40436697){\color[rgb]{0,0,0}\makebox(0,0)[lb]{\smash{$U_j^{\varepsilon}$}}}%
    \put(0.57921474,0.30249499){\color[rgb]{0,0,0}\makebox(0,0)[lb]{\smash{$U_j^{\varepsilon/2}$}}}%
    \put(0.33035615,0.00561092){\color[rgb]{0.81176471,0.2627451,0.94509804}\makebox(0,0)[lb]{\smash{$U_j$}}}%
    \put(0,0){\includegraphics[width=\unitlength,page=3]{phij.pdf}}%
  \end{picture}%
\endgroup%

%% file: dec2.pdf_tex
\begingroup%
  \makeatletter%
  \providecommand\color[2][]{%
    \errmessage{(Inkscape) Color is used for the text in Inkscape, but the package 'color.sty' is not loaded}%
    \renewcommand\color[2][]{}%
  }%
  \providecommand\transparent[1]{%
    \errmessage{(Inkscape) Transparency is used (non-zero) for the text in Inkscape, but the package 'transparent.sty' is not loaded}%
    \renewcommand\transparent[1]{}%
  }%
  \providecommand\rotatebox[2]{#2}%
  \ifx\svgwidth\undefined%
    \setlength{\unitlength}{261.14398158bp}%
    \ifx\svgscale\undefined%
      \relax%
    \else%
      \setlength{\unitlength}{\unitlength * \real{\svgscale}}%
    \fi%
  \else%
    \setlength{\unitlength}{\svgwidth}%
  \fi%
  \global\let\svgwidth\undefined%
  \global\let\svgscale\undefined%
  \makeatother%
  \begin{picture}(1,0.56096902)%
    \put(0,0){\includegraphics[width=\unitlength,page=1]{dec2.pdf}}%
    \put(0.18455774,0.01992243){\color[rgb]{0,0,0}\transparent{0.69803923}\makebox(0,0)[lb]{\smash{$\varepsilon/2$}}}%
    \put(0.38023187,0.07597889){\color[rgb]{0,0,0}\transparent{0.69803923}\makebox(0,0)[lb]{\smash{$\varepsilon/2$}}}%
    \put(0.47628522,0.28179329){\color[rgb]{0,0,0}\makebox(0,0)[lb]{\smash{$K'$}}}%
    \put(0.31456622,0.14585402){\color[rgb]{0,0,1}\makebox(0,0)[lb]{\smash{$\Psi=1$}}}%
    \put(0,0){\includegraphics[width=\unitlength,page=2]{dec2.pdf}}%
    \put(0.02131304,0.26733297){\color[rgb]{0.59607843,0,0}\makebox(0,0)[lb]{\smash{$\text{supp}\:\Psi$}}}%
    \put(0,0){\includegraphics[width=\unitlength,page=3]{dec2.pdf}}%
    \put(0.67014885,0.17536165){\color[rgb]{0,1,0}\makebox(0,0)[lb]{\smash{$\text{supp}\:\Phi$}}}%
    \put(0,0){\includegraphics[width=\unitlength,page=4]{dec2.pdf}}%
    \put(0.60553744,0.27796029){\color[rgb]{1,0.50196078,0}\makebox(0,0)[lb]{\smash{$\Phi=1-\Psi$}}}%
  \end{picture}%
\endgroup%

%% file: rkscucc1.pdf_tex
\begingroup%
  \makeatletter%
  \providecommand\color[2][]{%
    \errmessage{(Inkscape) Color is used for the text in Inkscape, but the package 'color.sty' is not loaded}%
    \renewcommand\color[2][]{}%
  }%
  \providecommand\transparent[1]{%
    \errmessage{(Inkscape) Transparency is used (non-zero) for the text in Inkscape, but the package 'transparent.sty' is not loaded}%
    \renewcommand\transparent[1]{}%
  }%
  \providecommand\rotatebox[2]{#2}%
  \ifx\svgwidth\undefined%
    \setlength{\unitlength}{291.39924903bp}%
    \ifx\svgscale\undefined%
      \relax%
    \else%
      \setlength{\unitlength}{\unitlength * \real{\svgscale}}%
    \fi%
  \else%
    \setlength{\unitlength}{\svgwidth}%
  \fi%
  \global\let\svgwidth\undefined%
  \global\let\svgscale\undefined%
  \makeatother%
  \begin{picture}(1,0.84153186)%
    \put(0,0){\includegraphics[width=\unitlength,page=1]{rkscucc1.pdf}}%
    \put(0.57095121,0.70718618){\color[rgb]{0.08235294,0,0}\makebox(0,0)[lb]{\smash{$2^{-j+1}$}}}%
    \put(0.52304543,0.58866348){\color[rgb]{0,0,0}\makebox(0,0)[lb]{\smash{$2^{-j-1}$}}}%
    \put(0,0){\includegraphics[width=\unitlength,page=2]{rkscucc1.pdf}}%
    \put(0.36295771,0.79818332){\color[rgb]{1,0,0}\makebox(0,0)[lb]{\smash{$M_k\cap E_j$}}}%
    \put(0,0){\includegraphics[width=\unitlength,page=3]{rkscucc1.pdf}}%
    \put(0.01586858,0.81248291){\color[rgb]{0,0,0}\makebox(0,0)[lb]{\smash{$E_j\cap K$}}}%
    \put(0,0){\includegraphics[width=\unitlength,page=4]{rkscucc1.pdf}}%
    \put(0.35258395,0.5345843){\color[rgb]{0,1,0}\makebox(0,0)[lb]{\smash{$\Delta$}}}%
    \put(0.59230181,0.54386893){\color[rgb]{0.78431373,0.38823529,0.76862745}\makebox(0,0)[lb]{\smash{$r_{k,j}(x)$}}}%
    \put(0.71441673,0.49880037){\color[rgb]{0,0,1}\makebox(0,0)[lb]{\smash{$x_{\perp}=r_k(x)$}}}%
    \put(0.32820019,0.38500318){\color[rgb]{0,0,1}\makebox(0,0)[lb]{\smash{$x_{\parallel}$}}}%
    \put(0.56767329,0.17582905){\color[rgb]{0,0,1}\makebox(0,0)[lb]{\smash{$x_{\parallel}$}}}%
    \put(0.26527129,0.03015899){\color[rgb]{0,0,1}\makebox(0,0)[lb]{\smash{$x_{\perp}$}}}%
    \put(0.42608254,0.21038039){\color[rgb]{0,0,1}\makebox(0,0)[lb]{\smash{$|x|$}}}%
    \put(0.70164957,0.44123272){\color[rgb]{0,0,1}\makebox(0,0)[lb]{\smash{$x\in E_j$}}}%
    \put(0,0){\includegraphics[width=\unitlength,page=5]{rkscucc1.pdf}}%
  \end{picture}%
\endgroup%

%% file: rkscucc2.pdf_tex
\begingroup%
  \makeatletter%
  \providecommand\color[2][]{%
    \errmessage{(Inkscape) Color is used for the text in Inkscape, but the package 'color.sty' is not loaded}%
    \renewcommand\color[2][]{}%
  }%
  \providecommand\transparent[1]{%
    \errmessage{(Inkscape) Transparency is used (non-zero) for the text in Inkscape, but the package 'transparent.sty' is not loaded}%
    \renewcommand\transparent[1]{}%
  }%
  \providecommand\rotatebox[2]{#2}%
  \ifx\svgwidth\undefined%
    \setlength{\unitlength}{343.02151298bp}%
    \ifx\svgscale\undefined%
      \relax%
    \else%
      \setlength{\unitlength}{\unitlength * \real{\svgscale}}%
    \fi%
  \else%
    \setlength{\unitlength}{\svgwidth}%
  \fi%
  \global\let\svgwidth\undefined%
  \global\let\svgscale\undefined%
  \makeatother%
  \begin{picture}(1,0.63770491)%
    \put(0,0){\includegraphics[width=\unitlength,page=1]{rkscucc2.pdf}}%
    \put(0.01139205,0.53507796){\color[rgb]{0.08235294,0,0}\makebox(0,0)[lb]{\smash{$2^{-j+1}$}}}%
    \put(-0.0017651,0.42216604){\color[rgb]{0,0,0}\makebox(0,0)[lb]{\smash{$2^{-j-1}$}}}%
    \put(0,0){\includegraphics[width=\unitlength,page=2]{rkscucc2.pdf}}%
    \put(0.38354958,0.59360717){\color[rgb]{1,0,0}\makebox(0,0)[lb]{\smash{$M_k\cap E_j$}}}%
    \put(0,0){\includegraphics[width=\unitlength,page=3]{rkscucc2.pdf}}%
    \put(0.07687073,0.61302761){\color[rgb]{0,0,0}\makebox(0,0)[lb]{\smash{$E_j\cap K$}}}%
    \put(0,0){\includegraphics[width=\unitlength,page=4]{rkscucc2.pdf}}%
    \put(0.3656685,0.2722469){\color[rgb]{0,0,1}\makebox(0,0)[lb]{\smash{$x_{\parallel}$}}}%
    \put(0.45104799,0.19741791){\color[rgb]{0,0,1}\makebox(0,0)[lb]{\smash{$|x|$}}}%
    \put(0,0){\includegraphics[width=\unitlength,page=5]{rkscucc2.pdf}}%
    \put(0.61855528,0.50390075){\color[rgb]{0,0,1}\makebox(0,0)[lb]{\smash{$x_{\perp}=r_k(x)=r_{k,j}(x)$}}}%
    \put(0.6146481,0.42263171){\color[rgb]{0,0,1}\makebox(0,0)[lb]{\smash{$x\in E_j$}}}%
    \put(0,0){\includegraphics[width=\unitlength,page=6]{rkscucc2.pdf}}%
  \end{picture}%
\endgroup%

%% file: aug7.pdf_tex
\begingroup%
  \makeatletter%
  \providecommand\color[2][]{%
    \errmessage{(Inkscape) Color is used for the text in Inkscape, but the package 'color.sty' is not loaded}%
    \renewcommand\color[2][]{}%
  }%
  \providecommand\transparent[1]{%
    \errmessage{(Inkscape) Transparency is used (non-zero) for the text in Inkscape, but the package 'transparent.sty' is not loaded}%
    \renewcommand\transparent[1]{}%
  }%
  \providecommand\rotatebox[2]{#2}%
  \ifx\svgwidth\undefined%
    \setlength{\unitlength}{213.26626165bp}%
    \ifx\svgscale\undefined%
      \relax%
    \else%
      \setlength{\unitlength}{\unitlength * \real{\svgscale}}%
    \fi%
  \else%
    \setlength{\unitlength}{\svgwidth}%
  \fi%
  \global\let\svgwidth\undefined%
  \global\let\svgscale\undefined%
  \makeatother%
  \begin{picture}(1,0.98800578)%
    \put(0,0){\includegraphics[width=\unitlength,page=1]{aug7.pdf}}%
    \put(0.0003391,0.00909176){\color[rgb]{0,0,0}\makebox(0,0)[lb]{\smash{$x_1$}}}%
    \put(0.84173796,0.15049672){\color[rgb]{0,0,0}\makebox(0,0)[lb]{\smash{$x_2$}}}%
    \put(0.37953905,0.94176122){\color[rgb]{0,0,0}\makebox(0,0)[lb]{\smash{$x_3$}}}%
    \put(0.65884883,0.72817131){\color[rgb]{1,0,0}\makebox(0,0)[lb]{\smash{$M_k$}}}%
    \put(0.33419768,0.48605761){\color[rgb]{0,0,0}\transparent{0.39215687}\makebox(0,0)[lb]{\smash{$x$}}}%
    \put(0.33181198,0.63781519){\color[rgb]{0,0,0}\transparent{0.39215687}\makebox(0,0)[lb]{\smash{$2^jx$}}}%
    \put(0.54592302,0.4577337){\color[rgb]{0,0,0}\transparent{0.39215687}\makebox(0,0)[lb]{\smash{$r_k(x)$}}}%
    \put(0.63775623,0.60630488){\color[rgb]{0,0,0}\transparent{0.39215687}\makebox(0,0)[lb]{\smash{$r_k(2^jx)$}}}%
    \put(0.48625978,0.29545396){\color[rgb]{0,0,0}\transparent{0.39215687}\makebox(0,0)[lb]{\smash{$\theta$}}}%
    \put(0,0){\includegraphics[width=\unitlength,page=2]{aug7.pdf}}%
  \end{picture}%
\endgroup%